\documentclass[a4paper,11pt]{amsart}
\usepackage{mathtools}
\usepackage{amsthm}
\usepackage{amsfonts}
\usepackage{amssymb}
\usepackage{bm}
\usepackage{latexsym}
\usepackage{mathrsfs}

\usepackage{enumitem}

\usepackage[section]{placeins}
\usepackage{xcolor}

\usepackage[abs]{overpic}
\usepackage[titletoc]{appendix}
\usepackage{graphicx}
\usepackage{latexsym}
\usepackage[utf8]{inputenc}
\usepackage{epsfig}
\usepackage{psfrag}

\usepackage{tikz}
\usepackage[normalem]{ulem}

\usepackage{diagbox}

\usepackage{thmtools}
\usepackage{thm-restate}

\usepackage[hidelinks]{hyperref}
\hypersetup{ 
  colorlinks   = true,            
  urlcolor     = {blue!50!black}, 
  linkcolor    = {blue!50!black}, 
  citecolor   = {red!50!black}    
}

\numberwithin{equation}{section}

\newtheorem{theorem}{Theorem}[section]{\bf}{\it}
\newtheorem{lemma}[theorem]{Lemma}{\bf}{\it}
\newtheorem{proposition}[theorem]{Proposition}{\bf}{\it}
\newtheorem{corollary}[theorem]{Corollary}{\bf}{\it}
{\bf}{\it} 
{\bf}{\it}
\newtheorem*{theorem*}{Theorem}

\newtheorem{remark}[theorem]{Remark}
\newtheorem*{remark*}{Remark}
{\bf}{\it}
{\bf}{\it}

\newtheorem{definition}[theorem]{Definition}
\newtheorem*{definition*}{Definition}

{\bf}{\it}
{\bf}{\it}
{\bf}{\it}

\newtheorem*{example*}{Example}

\theoremstyle{remark}

\theoremstyle{definition}

\theoremstyle{remark}

\newcommand{\diam}{\operatorname{diam}}

\newcommand{\Sbb}{\mathbb S}

\newcommand{\R}{\mathbb R}

\newcommand{\C}{\mathbb C}
\newcommand{\N}{\mathbb N}

\newcommand{\loc}{{\operatorname{loc}}}

\newcommand{\spt}{\operatorname{spt}}

\newcommand{\dist}{{\operatorname{dist}\,}}
\newcommand{\id}{{\operatorname{id}}}

\newdimen\vintkern\vintkern11pt
\def\vint{-\kern-\vintkern\int}

\newcommand{\dn}{\;\mathrm{d}}
\newcommand{\norm}[1]{\lVert #1 \rVert}

\newcommand{\grad}{\nabla}

\newcommand{\bS}{\mathbb{S}}

\newcommand{\cF}{\mathcal{F}}

\newcommand{\vol}{\mathrm{vol}}

\newcommand{\dR}{\mathrm{dR}}

\newcommand{\interior}{\mathrm{int}}

\newcommand{\cB}{\mathcal{B}}

\newcommand{\ol}{\overline}

\newcommand{\cU}{\mathcal{U}}

\newcommand{\set}[1]{\{\, {#1} \,\}}
\newcommand{\sset}[1]{\{{#1}\}}
\newcommand{\Set}[1]{\left\{\, {#1} \,\right\}}

\newcommand{\injrad}{\operatorname{inj\,rad}}

\renewcommand{\emptyset}{\varnothing}

\newcommand{\QR}{\mathsf{QR}}

\newcommand{\weakstarto}{\overset{*}{\rightharpoonup}}
\newcommand{\weakto}{\rightharpoonup}
\newcommand{\GHto}{\overset{GH}{\longrightarrow}}

\newcommand{\Sing}{\mathrm{Sing}}
\newcommand{\Strata}{\mathrm{Strata}}

\newcommand{\Cal}{\mathrm{Cal}}

\newcommand{\supp}{\mathrm{supp} \,}

\newcommand{\comass}{\mathrm{comass}}
\newcommand{\symp}{\mathrm{symp}}

\newcommand{\abs}[1]{\left\lvert #1 \right\rvert}

\DeclareMathOperator{\slim}{slim}
\DeclareMathOperator{\mlim}{mlim}

\renewcommand{\emptyset}{\varnothing}
\renewcommand{\le}{\leqslant}
\renewcommand{\ge}{\geqslant}
\renewcommand{\leq}{\leqslant}
\renewcommand{\geq}{\geqslant}
\renewcommand{\phi}{\varphi}
\renewcommand{\epsilon}{\varepsilon}

\title[]{Nodal resolution of quasiregular curves via bubble trees}

\author{Pekka Pankka}
\address{Department of Mathematics and Statistics, P.O. Box 68 (Pietari Kalmin katu 5), FI-00014 University of Helsinki, Finland}
\email{pekka.pankka@helsinki.fi}

\author{Jonathan Pim}
\address{Department of Mathematics and Statistics, P.O. Box 68 (Pietari Kalmin katu 5), FI-00014 University of Helsinki, Finland}
\email{jonathan.pim@helsinki.fi}

\thanks{This work was supported in part by the Academy of Finland project \#332671 and Center of Excellence FiRST}

\begin{document}

\begin{abstract}
We prove a version of Gromov's compactness theorem for quasiregular curves into calibrated manifolds with bounded geometry. In our main theorem, given an $n$-dimensional calibration $\omega$ on manifold $N$, we associate to a weak-$\star$ limit $\mu = \lim_{k \to \infty} \star F_k^*\omega$ of measures induced by a sequence $(F_k \colon X\to N)_{k\in \N}$ of $K$-quasiregular $\omega$-curves on a nodal manifold $X$, a bubble tree $\widehat X$ over \(X\), a sequence of mappings $(\widehat F_\ell \colon X \to N)_{\ell \in \N}$ converging locally uniformly to a quasiregular curve $\widehat F\colon \widehat X\to N$ which realizes the measure $\mu$, that is, $\mu = \pi_*(\star \widehat F^*\omega)$, where $\pi \colon \widehat X\to X$ is the natural projection. We call the sequence $(\widehat F_\ell)_{\ell \in \N}$ a nodal resolution of the sequence $(F_k)_{k\in \N}$.
As a corollary we obtain a normality criterion for families of quasiregular curves. Classic interpretations of bubbling via Gromov--Hausdorff convergence and pinching maps also follow.
\end{abstract}

\maketitle
 
{
    \hypersetup{linkcolor=black}
    \setcounter{tocdepth}{1}
    \tableofcontents
}

\section{Introduction}

It is a classical theorem in complex analysis that a locally uniform limit $f \colon \Sigma \to \Sigma'$ of a sequence of holomorphic maps $(f_k \colon \Sigma \to \Sigma')_{k\in \N}$ between closed Riemannian surfaces is holomorphic. It is similarly classical that the degree is preserved at the limit. More precisely, if $f$ is a non-constant, then $\deg f$ is the limit of the degrees $k\mapsto \deg f_k$ by Hurwitz's theorem. Gromov's compactness theorem \cite{Gromov-Invent-1985} gives a converse for this classical result: 
\begin{quote}
Let $(f_k \colon \Sigma \to \Sigma')$ be a sequence of holomorphic maps of uniformly bounded degree between Riemann surfaces. Then there exists a subsequence $(f_{k_j})$ of $(f_k)$ which converges in the nodal sense to holomorphic map $f \colon X\to \Sigma'$ from a nodal manifold $X$ over $\Sigma$. The limit of the degrees of maps $f_{k_j}$ is the sum of the degrees of the maps $f|_S \colon S\to \Sigma'$ for the strata $S$ of $X$.
\end{quote}
In \cite{Gromov-Invent-1985} Gromov showed more generally that, for pseudoholomorphic curves, area bounds yield sequential compactness if the local uniform convergence is understood in the more general sense of nodal convergence;
see e.g.~Pansu's \cite{Pansu-1994}, Wolfson \cite{Wolfson-1988}, Ruan--Tian \cite{Ruan-Tian-1995}, Hummel \cite{Hummel-book-1997}, and Bourgeois--Eliashberg--Hofer--Wysocki--Zehnder \cite{BEHWZ-2003} for proofs of Gromov's theorem for pseudoholomorphic curves.
The aforementioned nodal convergence results also hold in many other settings. For example, see Parker \cite{Parker-JDG-1996} for harmonic maps under energy bounds, Meier--Vikman--Wenger \cite{Meier-Vikman-Wenger-2025-bubbling-harmonic-2-spheres} for harmonic maps into metric spaces, Cheng--Karigiannis--Madnick \cite{Cheng-Karigiannis-Madnick} for associative Smith maps, Ivrii \cite{Ivrii} for meromorphic functions, and 
\cite{pankka-souto2023bubbleqr} for quasiregular mappings.

Recall that a continuous mapping $f\colon M \to N$ between oriented Riemannian $n$-manifolds is \emph{$K$-quasiregular for $K\ge 1$} if $f$ belongs to the Sobolev space $W^{1,n}_{\loc}(M,N)$ and satisfies the distortion inequality $\norm{Df}^n \le K J_f$ almost everywhere in $M$, where $\norm{Df}$ is the pointwise operator norm and $J_f = \star f^*\vol_N$ is the Jacobian determinant of the differential $Df$ of $f$. In this terminology, a $K$-quasiregular homeomorphism is called $K$-quasiconformal. See e.g.~Rickman's monograph \cite{Rickman-book-1993} for the theory of quasiregular mappings.

In this article, we prove Gromov's compactness theorem for quasiregular curves into calibrated manifolds. We recall the terminology after the following statement; for a statement without the compactness assumptions on the manifolds, see Theorem \ref{thm:main-reconstruction}.

\begin{theorem}
\label{thm:main-reconstruction-short}
Let $2\le n \le m$, let $X$ be a closed, connected, oriented Riemannian nodal $n$-manifold, let $(N,\omega)$ be a closed $n$-calibrated Riemannian $m$-manifold, and $K\ge 1$. Let $(F_k \colon X\to N)_{k\in \N}$ be a sequence of $K$-quasiregular $\omega$-curves for which $(\star F_k^*\omega)_{k\in \N}$ is a bounded sequence of measures in $X$. Then there exists a compact bubble tree \(\widehat X\) over \(X\), a sequence $(\widehat F_j \colon \widehat X\to N)_{j\in \N}$ of continuous $W^{1,n}_\loc(\widehat X,N)$ maps converging locally uniformly to a \(K\)-quasiregular \(\omega\)-curve \(\widehat F \colon \widehat X \to N\), and a subsequence \((F_{k_j})_{j\in \N}\) of \((F_k)_{k\in\N}\) for which 
\[
F_{k_j}|_{X\setminus P} \to \widehat F|_{X\setminus P} \text{ locally uniformly} \text{ and }
\star F_{k_j}^*\omega \weakstarto (\pi_{\widehat X,X})_*(\star \widehat F^*\omega) 
\]
as $j \to \infty$, where $P \subset X$ is a finite set. 
\end{theorem}

The interpretation of the limit map $\widehat F$ is twofold. On one hand, $\widehat F$ may be seen as a weak replacement of the classical locally uniform limit, made possible by enlarging the domain of the map. On the other hand, $\widehat F$ may be seen as a resolution of the singular parts of the limiting measure of the sequence $(\star F^*_{k_j}\omega)_{j\in \N}$.

We discuss now the terminology in Theorem \ref{thm:main-reconstruction-short} starting from quasiregular curves and calibrated manifolds, and then continuing to nodal manifolds and bubble trees.

For the definition of quasiregular curves, we recall first terminology related to calibrations.
Let $M$ be an oriented, Riemannian $n$-manifold and let $N$ be a Riemannian $m$-manifold for $2\le n \le m$. Given a closed non-vanishing smooth differential $n$-form $\omega \in \Omega^n(N)$ in $N$, we say that a continuous mapping $F\colon M \to N$ is a \emph{$K$-quasiregular $\omega$-curve} if $F$ belongs to the Sobolev space $W^{1,n}_\loc(M,N)$ and satisfies the distortion inequality
\begin{equation}
\label{eq:QRC}
(\norm{\omega}_\comass \circ F) \norm{DF}^n \le K (\star F^*\omega)
\quad \text{a.e.~in }M,
\end{equation}
where $\norm{\omega}_\comass$ is the pointwise comass norm function $p \mapsto \norm{\omega_p}_\comass$ of the form $\omega$ given by formula
\[
\norm{\omega_p}_\comass = \sup\{ \omega_p(v_1,\ldots, v_n) \colon |v_1|=\cdots = |v_n|=1\}
\]
for $p\in N$ and $\star \colon \Omega^n(M) \to \Omega^0(M)$ the Hodge star operator on $M$. 

The class of quasiregular curves contains as subclasses both quasiregular mappings and holomorphic curves. Indeed, for equidimensional oriented manifolds $M$ and $N$, $K$-quasiregular mappings $M\to N$ are $K$-quasiregular $\vol_N$-curves where \(\vol_N\) the Riemannian volume form of $N$. Similarly, holomorphic curves $\C \to \C^k$ are $1$-quasiregular $\omega_\symp$-curves, where $\omega_\symp$ is the standard symplectic form in $\C^k$. Relatedly, pseudoholomorphic curves for isometric almost complex structures are also $1$-quasiregular with respect to the associated symplectic forms; we refer to \cite{pankka2020qrcs} and \cite{Ikonen-Pankka-2024} for a more detailed discussion.
Furthermore, Smith maps, see e.g.~Cheng, Karigiannis, and Madnick \cite{Cheng-Karigiannis-Madnick, Cheng-Karigiannis-Madnick--variational} and Broder, Iliashenko, and Madnick \cite{Broder-Iliashenko-Madnick}, form another subclass of $1$-quasiregular curves in this terminology. 
Regarding convergence properties, we simply note that \emph{locally uniform limits of $K$-quasiregular $\omega$-curves are $K$-quasiregular $\omega$-curves}, see \cite{pankka2020qrcs}. 

In what follows, we consider only non-vanishing calibrations. A smooth form $\omega\in \Omega^n(N)$ is a \emph{non-vanishing $n$-calibration on $N$} if 
it is a closed differential $n$-form satisfying $0 <  \norm{\omega_p}_\comass \le \sup_{p\in N} \norm{\omega_p}_\comass = 1$.
We call the pair $(N,\omega)$ an \emph{$n$-calibrated $m$-manifold} and denote $\Cal^n_+(N)$ the space of strongly non-vanishing $n$-calibrations on $N$, that is, the space of non-vanishing \(n\)-calibrations \(\omega\) on $N$ having comass bounded away from zero i.e. 
$0 < \inf_{p\in N} \norm{\omega_p}_\comass \le \sup_{p\in N} \norm{\omega_p}_\comass = 1$. We refer to Harvey and Lawson \cite{harvey1982calibrated}, Joyce \cite{Joyce-book-2007}, or Morgan \cite[Chapter 6]{Morgan-book-2016} for detailed discussions about calibrations, and Heikkilä \cite{Heikkila-PLMS} and Ikonen \cite{ikonen2024remove} for discussions on non-vanishing conditions for calibrations in the context of quasiregular curves.

\newcommand{\assoc}{\mathrm{assoc}}

To our knowledge, the theorem of Cheng, Karigiannis, and Madnick \cite[Proposition 4.12 and Theorems 5.1 and 5.2]{Cheng-Karigiannis-Madnick} for associative Smith maps is the first Gromov's compactness theorem for higher dimensional calibrations. 

In the terminology above an associative Smith map is a $C^1$-smooth $1$-quasiregular $\omega_\assoc$-curve $\Sigma \to M$, where $\Sigma$ is an oriented Riemannian $3$-manifold and $M$ is a $G_2$-manifold calibrated by a closed associative form $\omega_\assoc\in \Omega^3(M)$. 
We refer to \cite{Cheng-Karigiannis-Madnick} for a discussion and the precise statement of this result and merely note that energy bounds for Smith maps cause bubbling, i.e.~nodal convergence, analogous to the case of pseudoholomorphic curves.

We say that a connected and second countable metrizable space $X$ is a \emph{nodal $n$-manifold} if there exists a discrete set $P \subset X$ having the property that 
closures of components of $X\setminus P$, i.e.~the strata of $X$, are $n$-manifolds, and each point of $P$ belongs to only finitely many closures of components of $X\setminus P$. We say that $p\in P$ is a \emph{singular point} if $p$ is contained in at least two strata and denote $\Sing(X)$ the set of all singular points. We denote $\Strata(X)$ the collection of all strata.
We also say that $X$ is a \emph{closed}, \emph{oriented}, \emph{smooth}, or \emph{Riemannian} nodal $n$-manifold if each stratum of $X$ is closed, oriented, smooth, or Riemannian, respectively. 

A subspace $X$ of a nodal $n$-manifold $\widehat X$ is a \emph{nodal submanifold of $\widehat X$} if $X$ is a nodal $n$-manifold for which $\Strata(X) \subset \Strata(\widehat X)$. Due to the connectedness of $X$ and $\widehat X$, $X$ is a retract of $\widehat X$, more precisely, there exists a unique extension $\pi_{\widehat X,X} \colon \widehat X\to X$ of the identity map $\id_X \colon X\to X$ having the property that $\pi_{\widehat X,X}(\widehat X\setminus X) \subset \Sing(X)$. We denote $\Sing_X(\widehat X)$ the singular points of $\widehat X$ which are not singular points of $X$ and $\Strata_X(\widehat X)$ the strata of $\widehat X$ not contained in $X$.

Following the usual terminology, we say that a nodal manifold $\widehat X$ is a \emph{bubble tree (over a nodal submanifold $X$)} if the elements of $\Strata_X(\widehat X)$ are topological $n$-spheres and the components of $\widehat X\setminus X$ are simply connected.
In what follows, we only consider oriented and Riemannian bubble trees over oriented and Riemannian nodal manifolds, and we omit these adjectives if there is no confusion.

Merging the terminology in \cite{pankka2020qrcs} and \cite{pankka-souto2023bubbleqr}, we say that a continuous mapping $F \colon X\to N$ from an oriented and Riemannian nodal $n$-manifold to a Riemannian $m$-manifold $N$ is a \emph{$K$-quasiregular $\omega$-curve for $K\ge 1$ and $\omega \in \Cal^n_+(N)$} if the restriction $F|_M \colon M\to N$ is a $K$-quasiregular $\omega$-curve for each stratum $M$ of $X$. We also interpret the pull-back $F^*\omega$ stratum-wise, that is, $F^*\omega|_M = (F|_M)^*\omega$ for each stratum $M$ of $X$; recall that $F$ is a Sobolev mapping in each stratum and that $F^*\omega$ is not pointwise defined even within a stratum. We have now introduced all the terminology in Theorem \ref{thm:main-reconstruction-short}.


Theorem \ref{thm:main-reconstruction-short} follows immediately from the following compactness theorem for quasiregular curves from non-compact spaces into spaces of bounded geometry.
Also in this statement, we interpret the convergence of a subsequence of the sequence $(F_k)_{k\in \N}$ in terms of the weak-$\star$ compactness of the measures $\star F^*\omega_k$. 
We give in Section \ref{sec:GH} two other interpretations, one in terms of the Gromov--Hausdorff convergence of Riemannian metrics of $M$ and the other in terms of the nodal convergence of pinching maps. 
The remaining terminology is introduced after the statement.

\begin{restatable}
{theorem}{MainTheReconstruction}
\label{thm:main-reconstruction}
Let $2\le n \le m$, let $X$ be a connected, oriented Riemannian nodal $n$-manifold, let $(N,\omega)$ be an $n$-calibrated Riemannian $m$-manifold having bounded geometry, and $K\ge 1$. Let $(F_k \colon X\to N)_{k\in \N}$ be a locally equibounded sequence of $K$-quasiregular $\omega$-curves for which there exists \(x_0 \in X\) such that the orbit \(\set{F_k(x_0) \mid k \in \N}\) has compact closure.
Then there exists a bubble tree \(\widehat X\) over \(X\), a sequence \((\widehat F_\ell \colon \widehat X \to N)_{\ell \in \N}\) of maps converging locally uniformly to a \(K\)-quasiregular \(\omega\)-curve \(\widehat F \colon \widehat X \to N\), and a subsequence \((F_{k_j})_{j\in \N}\) of \((F_k)_{k\in\N}\) satisfying the following conditions:
\begin{enumerate}
\item the maps \((F_{k_j}|_{X \setminus \Sing_X(\widehat X)})_j\) converge locally uniformly to \(\widehat F |_{X \setminus \Sing_X(\widehat X)}\); and
\item $\star F_{k_j}^*\omega \weakstarto (\pi_{\widehat X,X})_*(\star \widehat F^*\omega)$ as measures as $j \to \infty$.
\end{enumerate}
Moreover, for each $p\in \Sing_X(\widehat X)\cap X$, the pre-image $\pi_{\widehat X,X}^{-1}(p)$ consists of finitely many bubbles.
\end{restatable}

We say that a Riemannian manifold $N$ \emph{has bounded geometry} if the sectional curvature $\kappa$ of $N$ is bounded from above and below, that is, $|\kappa|\le \kappa_N <\infty$ and $N$ has injectivity radius bounded away from zero. In particular, closed Riemannian manifolds and open manifolds with cylindrical ends
have bounded geometry.

Finally, we say that a family $\cF$ of $W^{1,n}_\loc$-Sobolev maps $F \colon M \to N$ is \emph{locally equibounded} if there exists an open cover $\cU$ having the property that, for each $U\in \cU$, there exists a constant $C_U>0$ for which
\begin{equation}
\label{eq:loc-equibound}
\sup_{F\in \cF} \int_{U} \norm{DF}^n \le C_U.
\end{equation}
Note that, for a $K$-quasiregular $\omega$-curve $F\colon M \to N$, $\star F^*\omega \le \norm{DF}^n \le K (\star F^*\omega)$ almost everywhere and we may equivalently state \eqref{eq:loc-equibound} in terms of $\star F^*\omega$.

\subsection{Existence of bubbles; normal families}

One interpretation of Theorem \ref{thm:main-reconstruction} is that, by passing from the original domain $X$ of the sequence $(F_k \colon X\to N)_{k\in \N}$ to the bubble tree $\widehat X$, we find a sequence converging locally uniformly to a quasiregular $\omega$-curve $\widehat F \colon \widehat X\to N$. For a non-trivial bubble tree $\widehat X$, that is, a bubble tree having a bubble $S \in \Strata_X(\widehat X)$, restrictions of $\widehat F$ to bubbles yield non-constant quasiregular $\omega$-curves $\bS^n \to N$. 

For the forthcoming statements, we say that a calibration $\omega \in \Cal^n_+(N)$ is \emph{spherical} if there exists a smooth map $f\colon \bS^n \to N$ for which $\int_{\bS^n} f^*\omega \ne 0$.
In particular, for a non-spherical calibration $\omega \in \Cal^n_+(N)$, we have that $\widehat X=X$ and $F_{k_j} \to \widehat F$ in Theorem \ref{thm:main-reconstruction}. Thus, as a corollary, we have the following normality criterion for the family $\QR_{K,\omega}(X,N)$ of $K$-quasiregular $\omega$-curves $X\to N$.

\begin{corollary}
\label{cor:no-bubbling}
Let $2\le n \le m$ and let $(N,\omega)$ be an $n$-calibrated closed Riemannian $m$-manifold, where $\omega\in \Cal^n_+(N)$ is non-spherical. Then, for a connected and oriented Riemannian nodal $n$-manifold $X$ and $K\ge 1$, each locally equibounded family $\cF \subset \QR_{K,\omega}(X,N)$ is normal. Moreover, each $K$-quasiregular $\omega$-curve $X\to N$ is constant on the spherical strata of $X$.
\end{corollary}

\newcommand{\CP}{\mathbb{C}P}
\newcommand{\FS}{\mathrm{FS}}

As a particular application of Corollary \ref{cor:no-bubbling}, we consider quasiregular curves into the complex projective spaces $\CP^n$. Let $\omega = \frac{1}{k!}\omega_{FS}^{\wedge k}$ be a $(2k)$-calibration on $\CP^n$ for $k\in \{2,\ldots,n\},$ where $\omega_{\FS}$ is the Fubini--Study form of $\CP^n$. Since $H_{\dR}^2(\bS^{2k})=0$, the form $\omega$ is non-spherical and hence, all quasiregular $\omega$-curves $\bS^{2k}\to N$ are constant maps. However, there are -- for large enough $K\ge 1$ -- non-constant $K$-quasiregular $\omega$-curves $\R^{2k}\to \CP^n$. Indeed, by a construction of Luisto and Prywes \cite{Luisto-Prywes}, there are non-constant quasiregular mappings $\R^{2k} \to \CP^k$ and post-composing such maps with the inclusion $\CP^k \hookrightarrow \CP^n$ yields non-constant quasiregular $\omega$-curves $\R^{2k}\to \CP^n$. In particular, the family $\QR_{K,\omega}(\R^{2k}, \CP^n)$ is non-empty for $K\ge 1$ large enough. Since the calibration $\omega$ is non-spherical, we obtain from Corollary \ref{cor:no-bubbling} that the family $\QR_{K,\omega}(\R^{2k},\CP^n)$ is nevertheless small in the sense that, for each $K\ge 1$, a locally equibounded family $\cF$ in $\QR_{K,\omega}(\R^{2k}, \CP^n)$ is normal. We refer to Heikkil\"a \cite{heikkila2023rham} for a more detailed discussion on the existence of quasiregular curves.

\begin{remark}
By Corollary \ref{cor:no-bubbling}, having a non-spherical calibration is a sufficient condition for the normality of locally equibounded subfamilies of $\QR_{K,\omega}(X,N)$. We do not know if this is also a necessary condition. 
More precisely, given a spherical closed form $\omega\in \Cal^n_+(N)$, we do not know whether there exists 
a non-constant quasiregular $\omega$-curve $\bS^n \to N$.
\end{remark}

\subsection{Quasiconformality as an open condition}

The distortion condition \eqref{eq:QRC} is an open condition in the following sense:

\begin{quote}
For $(N,g_0)$ a closed Riemannian $m$-manifold, $\omega_0\in \Omega^n_+(N)$, and $1\le K < K'$, there exists a neighborhood $G$ of $g_0$ in the uniform topology of the Riemannian metrics of $N$ and neighborhood $U$ of $\omega_0$ in the uniform topology of $\Omega^n_+(N)$ having the following properties: each $K$-quasiregular $\omega$-curve $(M,g_M) \to (N,g)$, for $(g_M,\omega)\in G\times U$, is a $K'$-quasiregular $\omega_0$-curve $(M,g_M) \to (N,g_0)$.     
\end{quote}

This simple observation yields Gromov's compactness theorem for varying Riemannian and calibration structures on the target space. We formulate this corollary of Theorem \ref{thm:main-reconstruction} as follows.
As the result follows immediately from Theorem \ref{thm:main-reconstruction} and the aforementioned observation, we omit the details.

\begin{corollary}
Let $2\le n \le m$, let $X$ be an oriented Riemannian nodal $n$-manifold, let $N$ be a closed smooth $m$-manifold, and let $(g_k)_{k\in \N}$ be sequence of Riemannian metrics on $N$ converging uniformly to a Riemannian metric $g$ on $N$, and let $(\omega_k)_{k\in \N}$ be a sequence in $\Cal^n_+(N)$ converging uniformly to a differential $n$-form $\omega\in \Omega^n_+(N)$. Let also $K\ge 1$ and let $(F_k \colon X\to N)_{k\in \N}$ be a
locally equibounded sequence, where $F_k$ is a $K$-quasiregular $\omega_k$-curve from $M$ to $(N,g_k)$. Then there exists a bubble tree \(\widehat X\) over \(X\), a sequence \((\widehat F_\ell \colon \widehat X \to N)_{\ell \in \N}\) of maps converging locally uniformly to a \(K\)-quasiregular \(\omega\)-curve \(\widehat F \colon \widehat X \to N\), and a subsequence $(F_{k_j})_{j\in \N}$ of $(F_k)_{k\in\N}$ satisfying the following conditions: 
\begin{enumerate}
\item $F_{k_j}|_{X\setminus \Sing_X(\widehat X)} \to F|_{X\setminus \Sing_X(\widehat X)}$ locally uniformly, and
\item $\star F_{k_j}^*\omega \weakstarto (\pi_{\widehat X,X})_*(\star F^*\omega)$ as measures as $j \to \infty$.

\end{enumerate}
\end{corollary}

\subsection{Outline of the proof}

The method in \cite{Cheng-Karigiannis-Madnick} is not directly at our disposal as $K$-quasiregular curves, contrary to Smith maps, are not $n$-harmonic. It also differs from the method in \cite{pankka-souto2023bubbleqr}, though a pinching map interpretation is discussed in Section~\ref{sec:GH}.

As mentioned earlier in this introduction, local uniform limits $F\colon X\to N$ of $K$-quasiregular $\omega$-curves are also $K$-quasiregular $\omega$-curves and the associated measures $\star F^*\omega$ are weak-$\star$ limits of the corresponding measures of the curves in the sequence. Together with the removability of point singularities and a nodal surgery extension of the sequence at nodal points, we iteratively resolve the singular parts of the limiting measures. Indeed, as usual in proofs of Gromov's compactness theorem, given a sequence $(F_k \colon X\to N)_{k\in \N}$ of $K$-quasiregular $\omega$-curves satisfying local integrability bounds, the measures $\star F_k^*\omega$ have a weakly converging subsequence, whose singular part consists of discrete point masses. 

We resolve each point mass by attaching a spherical stratum $\bS^n$ to each of these singular points and replace the sequence $(F_k)$ by a new sequence $(F^1_k \colon X_1 \to N)_{k\in \N}$, where $X_1$ is the bubble tree over $X$ given by the new spheres. Then this process is repeated for the sequence $(F^1_k)$. Together with a renormalization of the new maps on the bubbles, this leads to an iterative construction of the bubble tree $X_1 \subset X_2 \subset \cdots \subset \widetilde X$ and a sequence of maps $\widetilde X\to N$ having a locally uniformly converging subsequence 
$(\widehat F_\ell \colon \widehat X \to N)_{\ell\in \N}$, where $\widehat X$ is obtained from $\widetilde X$ by removing the bubbles in high enough strata $\Strata_{X_k}(\widetilde X)$, where the convergence process has stopped, that is, in which the limiting map $\widehat F$ is constant.

An advantage of this strategy is that the resulting bubble trees have no necks and consist of spherical strata by construction. 
However, a new technical difficulty arises as the maps $X_i \to N$ in the new sequences are typically not quasiregular $\omega$-curves. Therefore, for the induction step of the proof,  
we introduce a notion of \emph{asymptotically $(K,\omega)$-quasiregular sequences} in Definition \ref{def:asymptotically-quasiregular} and give the iterative argument in terms of such sequences instead of sequences of quasiregular curves. We call sequences of maps from bubble trees $X_i$ \emph{nodal pre-resolutions of $(F_k)_{k\in \N}$} and the final sequence $(\widehat F_\ell \colon \widehat X \to N)_{\ell\in \N}$ a \emph{nodal resolution of $(F_k)_{k\in\N}$}.
We refer to Definition \ref{def:nodal-resolution} for the terminology and merely note here that the asymptotic formulation of Theorem \ref{thm:main-reconstruction} in terms of sequences $(\widehat F_\ell \colon \widehat X\to N)_{\ell\in \N}$ and $(F_{k_j} \colon X \to N)_{j\in\N}$ stems from the definition of nodal resolution.

\subsection*{Acknowledgements} We thank Eero Hakavuori, Susanna Heikkil\"a, and Toni Ikonen for discussions on the topics of the article.
We also thank Toni Ikonen for comments on a preliminary version of the article.


\section{Preliminaries}
\label{sec:preliminaries}

\newcommand{\BLrad}{\mathrm{BLrad}}
\newcommand{\QCrad}{\mathrm{QCrad}}

Regarding Riemannian manifolds, we use the following notations for injectivity radii. Let $M$ be a Riemannian $n$-manifold and $p\in M$. We denote $B_M(p,r)$ the metric ball of radius $r$ about $p$ in $M$. Similarly, we denote $B_{T_p M}(0,r)$ the metric ball in the tangent space $T_p M$. We also denote $\exp_p \colon T_p M\to M$ the exponential map at $p$. As usual the injectivity radius $\injrad_M(p)$ of $M$ at $p$ is the number
\[
\injrad_M(p) = \sup\{ r>0 \colon \exp_p|_{B_{T_p M}(0,r)} \text{ is injective}\}.
\]
Since the bilipschitz and quasiconformality constants of the restrictions $\exp_p|_{B_{T_p M}(0,r)} \colon B_{T_p M}(0,r) \to B_M(p,r)$ tend to $1$ as $r\to 0$, we may define, for $\varepsilon>0$, the bilipschitz injectivity radii $\BLrad_{\varepsilon}(M,p)$ and the quasiconformal injectivity radii $\QCrad_{\varepsilon}(M,p)$ of $M$ at $p$ by the formulas
\[
\BLrad_\varepsilon(M,p) = \sup\{ r>0 \colon \exp_p|_{B_{T_pM}(0,r)} \text{ is $(1+\varepsilon)$-bilipschitz}\}
\]
and
\[
\QCrad_\varepsilon(M,p) = \sup \{ r>0 \colon \exp_p|_{B_{T_pM}(0,r)} \text{ is $(1+\varepsilon)$-quasiconformal}\}.
\]
Recall that an $L$-bilipschitz homeomorphism is $L^{2n}$-quasiconformal.

If $M$ has bounded geometry, the bilipschitz and quasiconformal injectivity radii $\BLrad_\varepsilon(M,p)$ and $\QCrad_\varepsilon(M,p)$ are uniformly bounded from below away from zero. Indeed, suppose $N$ has sectional curvature $\kappa$ bounded by $\kappa_N$. Then the classical Rauch comparison theorem  yields, for each $\varepsilon>0$, a uniform radius $r_{N,\varepsilon}>0$, depending only on $\kappa_N$ and $\varepsilon>0$, for which the exponential mapping $\exp_p \colon T_pN \to N$ is $(1+\varepsilon)$-bilipschitz in a ball $B_{T_p}(0,r_{N,\varepsilon})$. See e.g.~monographs of Buser and Karcher \cite[Section 6.4]{Buser-Karcher-book} or Chavel \cite[Theorem IX.2.3]{Chavel-book} or e.g.~article of Dyer, Vegter, and Wintraecken \cite[Lemma 9]{Dyer-Vegter-Wintraecken}. We conclude that, under bounded geometry assumption, 
\[
\inf_{p\in M} \BLrad_{\varepsilon}(M,p) > 0 \quad \text{and} \quad \inf_{p\in M} \QCrad_{\varepsilon}(M,p)>0.
\]

\subsection{H\"older continuity and removability of point singularities}

We begin by recalling two analytic results on quasiregular curves of finite energy due to Ikonen: local H\"older continuity \cite{ikonen2023pushforward} and the removability of point singularities \cite{ikonen2024remove}.

We state two removability theorems, both of which follow from Ikonen's theorem \cite[Theorem 1.9]{ikonen2024remove}. The first is the removability of point singularities under the assumption of the existence of continuous extensions and the second is the removability of point singularities under local energy assumptions. 

\begin{theorem}[{\cite[Theorem 1.9]{ikonen2024remove}}]
\label{thm:cont-removability}
Let $2\le n \le m$, let $M$ be a Riemannian $n$-manifold, let $P \subset M$ be a discrete set, and let $(N,\omega)$ be an $n$-calibrated Riemannian $m$-manifold having bounded geometry. Then a continuous extension $M \to N$ of a $K$-quasiregular $\omega$-curve $M\setminus P \to N$ is a $K$-quasiregular $\omega$-curve.
\end{theorem}

\begin{theorem}[{\cite[Theorem 1.9]{ikonen2024remove}}]
\label{thm:Ikonen-removability}
Let $2\le n \le m$, let $M$ be a Riemannian $n$-manifold, let $P \subset M$ be a discrete set, let $(N,\omega)$ be an $n$-calibrated Riemannian $m$-manifold having bounded geometry, and let $F\colon M\setminus P \to N$ be a $K$-quasiregular $\omega$-curve satisfying $\int_{M\setminus P} \norm{DF}^n < \infty$. Then $F$ has a continuous extension $\widehat F \colon M \to N$, which is a $K$-quasiregular $\omega$-curve. 
\end{theorem}

Regarding H\"older continuity, we have the following version of Ikonen's H\"older continuity theorem for quasiregular curves with locally small energy. For the statement we define the following constant associated to H\"older bounds of quasiregular curves of small energy.

\begin{definition}
A constant $E_N>0$ is a \emph{small energy bound of an $n$-calibrated $m$-manifold $(N,\omega)$} if, for each $K\ge 1$, there exists constants $\alpha=\alpha(n,K,N,\omega)\in (0,1]$ and $C=C(n,K,N,\omega) \ge 1 $ having the following property: Each $K$-quasiregular $\omega$-curve $F\colon B^n \to N$ satisfying 
\[
\norm{DF}_{L^n} \le E_N
\]
is $\alpha$-H\"older in $B^n(1/2)$ with constant $C \norm{DF}_{L^n}$, that is, 
\[
d_N(F(x),F(y)) \le C \norm{DF}_{L^n}\norm{x-y}^\alpha 
\]
for $x,y\in B^n(0,1/2)$.
\end{definition}

In these terms, Ikonen's theorem states that a calibrated manifold $(N,\omega)$ of bounded geometry has a small energy bound. 
Note that the isoperimetric assumptions in \cite[Theorem 6.8]{ikonen2023pushforward} hold under our bounded geometry assumption; see Proposition 4.2 and discussion in Section 4 in 
\cite{ikonen2024remove}. 

\begin{theorem}[{\cite[Theorem 6.8]{ikonen2023pushforward}, \cite[Proposition 4.2]{ikonen2024remove}}]
\label{thm:Ikonen-Holder}
Let $2\le n \le m$, let $M$ be a Riemannian $n$-manifold, let $N$ be a Riemannian $m$-manifold having bounded geometry, and let $\omega \in \Cal^n_+(N)$ be a calibration. Then there exists a constant $E_N>0$ having the following property: Let $F\colon B^n \to N$ be a $K$-quasiregular $\omega$-curve satisfying 
\[
\norm{DF}_{L^n} = \left( \int_{B^n} \norm{DF}^n\right)^{1/n}\le E_N.
\]
Then
\[
d_N(F(x),F(y)) \le C \norm{DF}_{L^n}\norm{x-y}^\alpha 
\]
for $x,y\in B^n(0,1/2)$, where $C>0$ and $\alpha\in (0,1]$ depend only on $n$, $K$, $\omega$, and bounded geometry data of $N$. 
\end{theorem}

\begin{remark}
Note that no \emph{a priori} continuity of the quasiregular curve is needed for this H\"older continuity result. We refer to \cite{ikonen2023pushforward} for a detailed discussion. Additionally therein the small energy bound \(E_N\) is given explicitly using the isoperimetric assumptions on \(N\). Thus \(E_N\) depends only on the geometry of \(N\) and not on the form $\omega$.
\end{remark}

As a corollary, we obtain the local H\"older continuity of quasiregular curves $M\to N$ below the $2$-bilipschitz injectivity radius of the manifold \(M\) at a point. We also denote, for $\alpha>0$,
\[
|F|_{0,\alpha} = \sup_{\substack{x,y\in M \\ x\ne y}} \frac{d_N(F(x),F(y))}{d_M(x,y)^\alpha}
\]
the $\alpha$-H\"older norm of a quasiregular curve $F\colon M \to N$.

\begin{corollary}
\label{cor:local-holder-cont}
Let $2\le n \le m$, let $M$ be an oriented Riemannian $n$-manifold, let $(N,\omega)$ be an $n$-calibrated Riemannian $m$-manifold having bounded geometry, and let $E_N>0$ be a small energy bound for $N$.
Then a $K$-quasiregular $\omega$-curve $F\colon M \to N$ is locally H\"older continuous. More precisely, for $p\in M$ and $0<r<\BLrad_1(M,p)$ satisfying $\norm{DF}_{L^n(B_M(p,r))} \le E_N$, it holds
\[
|F_{|B_M(p,r/2)}|_{0,\alpha/4} \le C \norm{DF}_{L^n(B_M(p,r))},
\]
where $\alpha=\alpha(n,K,N,\omega)\in (0,1]$ and $C=C(n,K,N,\omega)\ge 1$.
\end{corollary}
\begin{proof}
It suffices to make the following observations. First, for $p\in M$, we may isometrically identify $T_p M$ with $\R^n$ and obtain a $2$-bilipschitz chart $\varphi = \exp_p^{-1}|_{B_M(p,r)} \colon B_M(p,r) \to B^n(0,r)$, which maps $B_M(p,r/2)$ onto $B^n(0,r/2)$. Second, the composition $F \circ \varphi^{-1} \colon B^n(0,r) \to N$ is a $4^n K$-quasiregular $\omega$-curve satisfying $\norm{D(F \circ \varphi^{-1})} \le 2 \norm{DF}$ almost everywhere in $B^n(0,r)$. 
\end{proof}


\subsection{Precomposition of curves by quasiregular mappings}

As a preliminary result, we show that the pre-composition of a quasiregular 
$\omega$-curve with a quasiregular mapping is a quasiregular $\omega$-curve. This is a counterpart of the classical result that the composition of quasiregular mappings is quasiregular. We adapt the analytic proof of Bojarski and Iwaniec \cite{bojarski1983analytical}.
We shall only need this for the case of \(1\)-quasiconformal mappings between nodal manifolds, however we state the general version for future convenience. The nodal manifold version is obtained by applying the result to each strata.

\begin{theorem}
\label{lem:pre-comp-qrc-with-qr}
  Let \(N\), \(M\), and \(\Sigma\) be connected and oriented Riemannian manifolds of dimension \(m\), \(n\) and \(n\) respectively, and let \(\omega \in \Omega^n(N)\) be a closed, non-vanishing \(n\)-form, for \(m \geq n \geq 2\).
If \(\phi\colon M \to \Sigma\) is a \(K_{\phi}\)-quasiregular mapping and \(F\colon \Sigma \to N\) is a \(K\)-quasiregular \(\omega\)-curve, then \(F\circ \phi \colon M \to N\) is a \((K_{\phi}K)\)-quasiregular \(\omega\)-curve.
\end{theorem}

We begin by showing that the composition is in the correct Sobolev space.

\begin{lemma}\label{lem:pushforward-sobolev-qr}
  Let \(N\), \(M\), and \(\Sigma\) be connected and oriented Riemannian manifolds of dimension \(m\), \(n\) and \(n\) respectively.
  If \(\phi \colon M \to \Sigma\) is a quasiregular mapping and \(F \in W^{1,n}_{\text{loc}}(\Sigma, N)\), then \(F \circ \phi \in W^{1,n}_{\text{loc}}(M,N)\) and \(D(F \circ \phi) = DF \circ D\phi\) almost everywhere in \(M\).
\end{lemma}
\begin{proof}
  Let \(\iota\colon N \to \R^{n^{\prime}}\) be a Nash embedding, and let \((V, \psi)\) and \((U, \sigma)\) be \(2 \)-bilipschitz charts in \(\Sigma\) and \(M\) respectively for which \(\varphi( U) \subset V\).
  Thus \(\widetilde{\varphi} = \psi \circ \varphi \circ \sigma^{-1} \in W^{1,n}_{\text{loc}}(\sigma U,\psi V)\) and \(\iota \circ F \in W^{1,n}_{\text{loc}}(\Sigma,\R^{n^{\prime}})\) which implies \(\widetilde{F} = \iota \circ F \circ \psi^{-1} \in W^{1,n}_{\text{loc}}(\psi V,\R^{n^{\prime}})\).
Since a composition of a quasiregular mapping with a bilipschitz mapping is quasiregular, the map $\widetilde \phi$ is quasiregular. 
Then by \cite[Lemma~9.6]{bojarski1983analytical}, the fact that \(\psi\) and \(\sigma \) are bilipschitz, and by \cite[Corollary 5.5 and Lemma 6.3]{kangasniemi2021notes}, 
we have that
\(\widetilde{F} \circ \widetilde{\phi} = \iota \circ F \circ \phi \circ \sigma^{-1}  \in W^{1,n}_{\text{loc}}(\sigma U, \R^{n^{n\prime}})\),
\(D(\widetilde{F} \circ \widetilde{\phi}) = D\widetilde{F} \circ D\widetilde{\phi} = D\iota \circ DF \circ D\phi \circ D\sigma^{-1}\), and
\(D(\widetilde{F}\circ \widetilde{\phi}) = D(\iota\circ F\circ \phi \circ \sigma^{-1}) = D\iota \circ D(F\circ \phi) \circ D\sigma^{-1}\).
Hence, \(D(F\circ \phi) = DF \circ D\phi\) almost everywhere.
\end{proof}

\begin{proof}[Proof of Proposition~\ref{lem:pre-comp-qrc-with-qr}]
  Lemma~\ref{lem:pushforward-sobolev-qr} implies that \(F\circ \phi \in W^{1,n}_{\text{loc}}(M,N)\) and \(D(F\circ \phi) = DF \circ D\phi\) almost everywhere in \(M\). So it suffices to show that \(F\circ \phi\) satisfies the distortion inequality.

We have that \((\norm{\omega} \circ F)\norm{DF}^n \leq K\star F^{ * }\omega\) almost everywhere in \(\Sigma\), and \(\norm{D\phi}^n \leq K_{\phi} (\star\phi^{ * }\vol_{\Sigma}) \) almost everywhere in \(M\). Since \(\phi\) satisfies the Lusin \((N^{-1})\) condition,  we obtain that
\((\norm{\omega}\circ F)_{\phi(x)}\norm{DF_{\phi(x)}}^n \leq K \star (F^{ * }\omega)_{\phi(x)}\) for almost every \(x \in M\). 
Thus
\begin{align*}
(\norm{\omega}\circ (F\circ \phi))_x\norm{D(F\circ \phi)_x}^n
  &\leq (\norm{\omega}\circ F)_{\phi(x)}\norm{DF_{\phi(x)}}^n \norm{D\phi_x}^n \\
  &\leq (K K_{\phi})\left((\star F^{ * }\omega) \circ \phi\right) (\star\phi^{ * }\vol_{\Sigma}) \\
  &= (K K_{\phi})\left(\star( \phi^* (\star F^{ * }\omega) \phi^{ * }\vol_{\Sigma})\right) \\
  &= (K K_{\phi})\left(\star( \phi^* \left( (\star F^{ * }\omega) \vol_{\Sigma})\right)\right) \\
  &= (K K_{\phi})\left(\star( \phi^* (F^{ * }\omega)\right) \\
  &= (K K_{\phi})\left(\star( (F\circ \phi)^* \omega)\right).
\end{align*}
This concludes the proof.
\end{proof}


\subsection{Weak compactness of measures}

For the reader's convenience, we recall here the following weak compactness theorem for Radon measures. This follows from the the Banach--Alaoglu Theorem by a standard diagonal argument and the Riesz Representation Theorem; see {\cite[Theorem~4.4]{Simon-geo-meas-book-1983}}. We shall apply this to the measures \(\mu = \star F^* \omega\) induced by quasiregular $\omega$-curves $F\colon M\to N$.

\begin{theorem}[Weak Compactness of Radon Measures]
\label{thm:weak-compact-radon-measures}
Let \(X\) be a locally compact and \(\sigma\)-compact Hausdorff space. Suppose \((\mu_k)_{k\in\N}\) is a sequence of Radon measures satisfying \(\sup_{k \in \N}\mu_k(K) < \infty\) for all compact subsets \(K \subset X\). Then \((\mu_k)_{k\in\N}\) has a subsequence \((\mu_{k_j})_{j\in\N}\) and there exists a Radon measure \(\mu\) on \(X\) such that \(\mu_{k_j} \weakto \mu\) in the weak-\(\star\) topology. Moreover, if $X$ is compact, then $\mu(X) = \limsup_{j\to \infty} \mu_{k_j}(X)$.
\end{theorem}

\subsection{Nodal manifolds}

We discuss some preliminaries related to nodal manifolds. Recall, from the introduction, that a metrizable space $X$ is a nodal $n$-manifold if $X = \bigcup_{\alpha \in \Lambda} M_\alpha$, where each $M_\alpha$ is an $n$-manifold and, for $\alpha\ne \beta$, the intersection $M_\alpha\cap M_\beta$ is either a singleton or the empty set; such a manifold $M_\alpha$ is called a \emph{stratum of $X$} and the points $p\in X$, which are shared by at least two different strata, are called the singular points of $X$. The set $\Sing(X)$ of singular points of $X$, and hence also the set $\Strata(X)$ of all strata of $X$, is uniquely determined.

In what follows, we assume that a nodal $n$-manifold $X$ is connected, oriented, and Riemannian, in the sense that the strata of $X$ are oriented and Riemannian. Then each stratum $M\in \Strata(X)$ has Riemannian distance $d_M$. To define a length metric on $X$, we say that a path $\gamma \colon [0,1]\to X$ is \emph{admissible} if there exists an (admissible) partition $0=s_0 < \cdots < s_k = 1$ of $[0,1]$ for which $\gamma([s_{j-1},s_j])$ is contained in a single stratum $M_j$ of $X$. The length $\ell(\gamma)$ of an admissible path is
\[
\ell(\gamma) = \sup_{0=s_1 <\cdots < s_k =1} \sum_{j=1}^k \ell_{M_j}(\gamma|_{[s_{j-1},s_j]}),
\]
where the supremum is taken over all admissible partitions of $\gamma$; as usual, all admissible partitions yield the same sum of lengths. The length metric $d_X \colon X\times X\to [0,\infty)$ on $X$ is now given by formula
\[
d_X(p,q) = \inf_\gamma \ell(\gamma),
\]
where the infimum is taken over admissible paths $\gamma \colon [0,1]\to X$ from $p$ to $q$. In what follows, we denote, for $A\subset X$ and $r>0$, 
\[
  B_X(A, r) = \{ x \in X \colon {\dist}_X(x,A)<r\}
\]
the \emph{$r$-neighborhood of $A$ in $X$}.

As already mentioned in the introduction, we mainly consider bubble trees over nodal manifolds. More precisely, a bubble $n$-tree $\widehat X$ is a nodal $n$-manifold for which there exists a nodal $n$-submanifold $X \subset \widehat X$ having the property that strata of $\widehat X$ which are not contained in $X$, that is, elements of $\Strata_X(\widehat X)$, are $n$-spheres. We call a stratum $S$ in $\Strata_X(\widehat X)$ a \emph{bubble}.
This notation reflects the property that, in what follows, we equip each bubble with a Riemannian metric making it isometric to a copy of $\bS^n_p := \bS^n \times \{p\}$ of the standard Euclidean unit sphere $\bS^n$ attached to $X$ by identifying $p$ and the south pole $-e_{n+1}$ of $\bS^n$. In what follows, we also call the ball $\bS^n_{p,+} = \{ (x_1,\ldots, x_{n+1},p) \in \bS^n\times \{p\} \colon x_{n+1}\ge 0\} \subset \bS^n_p$, the \emph{upper hemisphere $\bS^n_{p,+}$ of a bubble $\bS^n_p$}.

\subsection{Quasiregular curves from nodal manifolds}

As in the introduction, we define differential forms on nodal manifolds stratum-wise. More formally, we call an indexed family $\omega = (\omega_M)_{M\in \Strata(X)}$ a \emph{differential $n$-form on $X$} if each $\omega_M$ is a differential $n$-form $M$. Note that, at a nodal point $p\in \Sing(X)$, the tangent spaces $T_p M$, for strata $M$ containing $p$, are distinct and therefore the covectors $(\omega_M)_p \colon T_p M \times \cdots \times T_p M \to \R$ are formally distinct. 
In what follows, we denote the form $\omega_M$ in the family $\omega$ simply as $\omega|_M$.

As in the introduction, we also say that a continuous mapping $F\colon X\to N$ from a connected, oriented, and Riemannian nodal $n$-manifold to a Riemannian $m$-manifold, for $2\le n \le m$, is a \emph{(nodal) $K$-quasiregular $\omega$-curve for $K\ge 1$ and $\omega \in \Cal^n_+(N)$} if, for each stratum $M\in \Strata(X)$, the mapping $F|_M \colon M \to N$ is a $K$-quasiregular $\omega$-curve.

Following the convention above, we denote $F^*\omega = ((F|_M)^*\omega)_{M\in \Strata(X)}$ and use the notation
\[
\int_X F^*\omega = \sum_{M\in \Strata(X)} \int_M (F|_M)^*\omega.
\]
Note that the sum is defined, since the functions $\star (F|_M)^*\omega$, and hence also their integrals, are non-negative. We extend this notational convention to all continuous $W^{1,n}_\loc$-mappings $f \colon X\to N$ for which the integrals $\int_M (f|_M)^*\omega$, $M\in \Strata(X)$, are non-negative.

\section{Asymptotically quasiregular sequences}

Later when we extend sequences of quasiregular curves to a bubble tree by modifying them in the neighborhood of a singular set, the maps in the resulting sequence will unfortunately, in general, no longer satisfy the distortion inequality in those neighborhoods. However with this operation, called nodal surgery (defined in Section~\ref{sec:nodal-surgery}) we may choose the neighborhoods of non-quasiregularity in such a way that they shrink to the singular set in the limit. In this sense this new sequence from the bubble tree is asymptotically quasiregular. We formalize this idea in Definition~\ref{def:quasiregular-exhaustion} and Definition~\ref{def:asymptotically-quasiregular}.
For these reasons in what follows we work with these intermediate classes of objects instead of sequences of quasiregular curves.

\begin{definition}
\label{def:quasiregular-exhaustion}
Let $X$ be an oriented and Riemannian nodal $n$-manifold and let $(N,\omega)$ be an $n$-calibrated Riemannian $m$-manifold for $2 \le n \le m$. A sequence $(F_k \colon X\to N)_{k\in \N}$ of continuous $W^{1,n}_\loc$-mappings admits a \emph{$(K,\omega)$-quasiregular exhaustion about a discrete set $Q \subset X$} if there exists a sequence $\varepsilon_k \downarrow 0$ having the following properties: for each neighborhood $U$ of $Q$, there exist $k_0\in \N$ and a neighborhood $V \subset \overline{V} \subset U$ of $Q$ for which $F_k|_{X\setminus \overline{V}} \colon X\setminus \overline{V}\to N$ is a $(1+\varepsilon_k)K$-quasiregular $\omega$-curve for $k\ge k_0$.
\end{definition}

\begin{definition}
\label{def:asymptotically-quasiregular}
Let $X$ be an oriented and Riemannian nodal $n$-manifold and let $(N,\omega)$ be an \(n\)-calibrated Riemannian $m$-manifold for $2 \le n \le m$. A sequence $(F_k \colon X\to N)_{k\in \N}$ of continuous $W^{1,n}_\loc$-mappings is \emph{asymptotically $(K,\omega)$-quasiregular} for $K\ge 1$ if there exists mutually disjoint discrete subsets $Q$ and $P$ of $X$ having the following properties: 
\begin{enumerate}
\item the sequence $(F_k)_{k\in \N}$ admits a $(K,\omega)$-quasiregular exhaustion about $Q$; and  \label{item:asymp-quasi-1}
\item the sequence $(F_k|_{X\setminus P} \colon X\setminus P \to N)_{k\in \N}$ converges locally uniformly.
\label{item:asymp-quasi-2}

\end{enumerate} 
\end{definition}

In what follows, we restrict ourselves to asymptotically quasiregular sequences $(F_k \colon X\to N)_{k\in \N}$ for which $Q$ and $P$ are subsets of $\Sing(X)$ and $X\setminus \Sing(X)$, respectively.


Since locally uniform limits of \(K\)-quasiregular \(\omega\)-curves are $K$-quasiregular $\omega$-curves, asymptotically quasiregular sequences converge locally uniformly, outside a discrete set $P$ as in Definition \ref{def:asymptotically-quasiregular}, to quasiregular curves. We record this as follows; see also \cite[Theorem~1.9 and Lemma~4.4]{pankka2020qrcs}.

Additionally for the proof, we say that a sequence $(F_k \colon X\to N)_{k\in \N}$ is \emph{quasiregular at a point $p\in X$} if there exists $k_0\in \N$ and a neighborhood $V\subset X$ of $p$ for which the restriction $F_k|_V \colon V \to N$ is a quasiregular curve for $k\ge k_0$.

\begin{lemma}\label{lem:loc-unif-limit-energy-bound}
Let $2\le n \le m$, let $X$ be an oriented and Riemannian nodal $n$-manifold, and let $(N,\omega)$ be an \(n\)-calibrated Riemannian $m$-manifold.
Let $(F_k \colon X \to N)_{k\in \N}$ be an asymptotically $(K,\omega)$-quasiregular sequence with \((F_k|_{X \setminus P})_k\) converging locally uniformly where \(P \subset X\) is a discrete set.
Then the sequence $(F_k|_{X\setminus P} \colon X\setminus P \to N)_{k\in \N}$ converges locally uniformly to a $K$-quasiregular $\omega$-curve $F \colon X \setminus P \to N$ which satisfies
  \begin{equation}\label{eq:7}
    \int_E\norm{DF}^n \leq \liminf_{k\to\infty}\int_E \norm{DF_k}^n
  \end{equation}
  for each compact subset $E \subset X$.
\end{lemma}
\begin{proof}
Let \(Q \subset X\setminus P\) be the discrete set about which \((F_k)_k\) admits a \((K,\omega)\)-quasiregular exhaustion. Set \(\widetilde P = Q \cup P\).
By definition, the sequence \( (F_k|_{X \setminus P} \colon X \setminus P \to N )_{k\in\N}\) is quasiregular at each point in $X\setminus Q$ and converges locally uniformly to a continuous map
\(F \colon X \setminus P \to N\).
 Observe that the same holds when \(P\) is replaced by \(\widetilde P\).
By \cite[Theorem~1.9]{pankka2020qrcs}, \(F|_{X \setminus \widetilde P}\) is a \(K\)-quasiregular \(\omega\)-curve, which -- by Theorem~\ref{thm:cont-removability} -- extends over the discrete set \(Q\) to a \(K\)-quasiregular \(\omega\)-curve \(F\colon X \setminus P \to N\).

Let \(E \subset X\) be a compact set. First consider the special case that \(E\) is contained in a single stratum \(M \in \Strata(X)\). 
We split this case into two cases: (i) \(E \cap \widetilde P = \emptyset\), and (ii) \(E \cap \widetilde P \neq \emptyset\).
The general case follows from the observation that $E\setminus \Sing(X)$ meets only finitely many strata.

Case (i): Since \(E \cap \widetilde P = \emptyset\), there exists a neighborhood \(U\) of \(\widetilde P\) for which \(E \cap U = \emptyset\). Then there exists \(k_0 \in \N\) and a neighborhood \(V \subset \overline V \subset U\) of \(\widetilde P\) such that for all \(k \geq k_0\), \(F_k|_{M\setminus \overline V}\) is a \((1+\epsilon_k) K\)-quasiregular \(\omega\)-curve.
Let \(\epsilon > 0\) be small enough that \(B_M(E,\epsilon) \subset M \setminus \overline V\).
Let \(\cU\) be the collection of closed balls \(\ol{B}_M(p,r)\) with \(p \in E\) and \(2r \leq \epsilon\), for which \(B(p,2r)\) is contained in a \((1+\epsilon)\)-bilipschitz chart at \(p\) and \(F(B_M(p,2r))\) is contained in a \((1+\epsilon)\)-bilipschitz chart at \(F(p)\). 

Since Riemannian manifolds are Vitali spaces \cite[Theorem~3.4.3]{HKSTsobolevSpaces}, \(\cU\) contains a countable and pairwise disjoint subfamily \(\cU_0 \subset \cU\) for which \(\vol_M\left(E\setminus \bigcup_{B \in \cU_0}B\right) = 0\).

For each \(B \in \cU_0\), let \(\phi_B\) and \(\psi_B\) be \((1+\epsilon)\)-bilipschitz charts containing \(B\) and \(F(B)\cup \bigcup_{k \ge k_0} F_k(B)\), respectively. Thus the sequence \((\psi_B\circ F_k \circ \phi_B^{-1})_{k\ge k_0}\) converges locally uniformly to $\psi_B\circ F \circ \phi_B^{-1}$. By \cite[Theorem~1.9]{pankka2020qrcs}, $\psi_B\circ F \circ \phi_B^{-1}$ is a $(1+\varepsilon)^{4n}K$-quasiregular $\omega$-curve. 
  Furthermore, by a change of variables, \cite[Lemma~4.4]{pankka2020qrcs}, and Fatou's lemma,
\begin{align*}
\int_{E}\norm{DF}^n
&= \sum_{B \in \cU_0}\int_B\norm{DF}^n \leq \sum_{B \in \cU_0}(1+\epsilon)^{3n}\int_{\phi_B(B)}\norm{D(\psi_B \circ F \circ \phi_B^{-1})}^n \\
&\leq (1+\epsilon)^{3n}\sum_{B \in \cU_0} \liminf_{k\to\infty}\int_{\phi_B(B)}\norm{D(\psi_B \circ F_k \circ \phi_B^{-1})}^n \\
&\leq (1+\epsilon)^{6n}\sum_{B \in \cU_0} \liminf_{k\to\infty}\int_{B}\norm{D F_k}^n 
\leq (1+\epsilon)^{6n} \liminf_{k\to\infty}\int_{E}\norm{D F_k}^n.
\end{align*}
Case (i) now follows.

For case (ii), i.e.~\(E \cap \widetilde P \neq \emptyset\), note that since \(\widetilde P\) is discrete, it suffices to consider the further special case: \(E = \bar B(p,R)\) and \(E \cap \widetilde P = \sset{p}\) . Consider the closed annuli \(E_r = E \setminus B(p,r)\) for \(0 < r < R\). Observe that each \(E_r\) is compact and does not meet \(\widetilde P\), so case (i) implies that
\[ \int_{E_r}\norm{DF}^n \leq \liminf_{k\to\infty}\int_{E_r}\norm{D F_k}^n \leq \liminf_{k\to\infty}\int_{E\setminus \{p\}}\norm{D F_k}^n = \liminf_{k\to\infty}\int_{E}\norm{D F_k}^n\]
for all \(0<r<R\). This case follows from observing that 
\[
\int_{E_r}\norm{DF}^n \to \int_{E}\norm{DF}^n
\] 
as \(r \to 0\). The claim now follows.
\end{proof}

\begin{remark}
According to classical terminology, Lemma \ref{lem:loc-unif-limit-energy-bound} yields that the sequence $(F_k \colon X\to Y)_{k\in \N}$ is \emph{quasinormal}. Recall that a family $\cF$ of maps $X\to Y$ is \emph{quasinormal} if every sequence $(f_k \colon X\to Y)$ in $\cF$ has a subsequence $(f_{k_j})$ which converges locally uniformly in $X\setminus P$, where $P$ is a discrete set. See e.g.~Nevo, Pang, and Zalcman \cite{Pang-Nevo-Zalcman} for discussion on quasinormality for meromorphic functions.
\end{remark}

We finish this section by recording an Arzel\`a--Ascoli type normality criterion for sequences admitting a quasiregular exhaustion. Recall that, by Corollary \ref{cor:local-holder-cont}, a family $\{F_k \colon X\to N\}_{k\in \N}$ of $K$-quasiregular $\omega$-curves, having small locally equibounded energy, is locally uniformly $\alpha$-H\"older continuous, with $\alpha$ depending only on the data. In particular, such a family is locally equicontinuous, and a fortiori normal by a standard Arzel\`a--Ascoli argument. 

In the forthcoming sections, we construct locally equicontinuous families of Sobolev mappings $\{F_k \colon X\to N\}_{k\in \N}$ admitting a quasiregular exhaustion about a discrete set. For this reason, we immediately state here a version of the Arzel\`a--Ascoli theorem specialized to this setting.

\begin{proposition}
\label{prop:norm-fams-hold-qrcs}
Let $2\le n \le m$, let $X$ be an oriented and Riemannian nodal $n$-manifold, and let $(N,\omega)$ be a complete \(n\)-calibrated Riemannian $m$-manifold.  
Let $(F_k \colon X\to N)_{k\in \N}$ be a sequence of continuous $W^{1,n}_\loc$-Sobolev mappings admitting a $(K,\omega)$-quasiregular exhaustion about a discrete set $Q\subset X$ 
and let \(P \subset X \setminus Q\) be another discrete set.
Then \((F_k)_{k\in\N}\) satisfies the following two conditions:
\begin{enumerate}
\item the family $\{F_k \}_{k\in \N}$ is locally equicontinuous in \(X \setminus P\); and \label{item:N1}
\item in each component \(C\) of \(X \setminus P\) there exists a point \(x_C \in C\) for which the orbit $\{F_k(x_C) \colon k\in \N\}$ has compact closure in $N$, \label{item:N2}
\end{enumerate}
if and only if there exists an asymptotically $(K,\omega)$-quasiregular subsequence \((F_{k_j})_{j \in \N}\) of \((F_k)_{k\in \N}\) converging locally uniformly on \(X \setminus P\). 
\end{proposition}

\begin{proof}
We begin by assuming that \((F_k)_{k\in\N}\) satisfies \eqref{item:N1} and  \eqref{item:N2}.
Assumption \eqref{item:N1} implies that there exists a covering $\cU$ of $X\setminus P$ for which \((F_k|_U)_k\) is equicontinuous on each element $U\in \cU$. By passing to components of elements of $\cU$, we may assume that elements of $\cU$ are connected.

By assumption \eqref{item:N2}, we may fix for each component $C$ of $X\setminus P$, an element $U_C$ of $\cU$ containing $x_C$. 
As \(N\) is complete, the equicontinuity of \((F_k|_{U_C})_k\) and Hopf--Rinow theorem give that, for each \(x \in U_C\), the orbit \(\set{F_{k}(x) \mid k \in \N}\) has compact closure in \(N\).
Thus, by the Arzel\`a--Ascoli theorem, there is a subsequence \((F_{k_j}|_{U_C})_j\) of \((F_{k}|_{U_C})_k\) which converges locally uniformly on \(U_C\). 

Since we may connect an element $U$ of $\cU$, contained in a component, say $C$, of $X\setminus P$ to an element $U_C$ by a finite chain of elements in $\cU$ contained in $C$, a standard diagonal argument yields a locally uniformly converging subsequence, also denoted $(F_{k_j})_{j\in \N}$ of $(F_k)_{k\in \N}$.
Observe that \((F_{k_j})_{k_j}\) still admits a \((K,\omega)\)-quasiregular exhaustion about \(Q\).
Therefore \((F_{k_j})_j\) is an asymptotically \((K, \omega)\)-quasiregular subsequence of \((F_k)_k\).

Now, for the converse direction, assume that \((F_k)_{k\in\N}\) has an asymptotically $(K,\omega)$-quasiregular subsequence \((F_{k_j})_{j \in \N}\) which converges locally uniformly on \(X \setminus P\). Then each component \(C\) of \(X \setminus P\) has an open cover \(\cU_C\) with the property that \((F_{k_j}|_U)_{j\in\N}\) converges locally uniformly on each \(U \in \cU_C\). 
Take \(U \in \cU_C\).
Thus by the Arzel\`a--Ascoli theorem, \((F_k|_U)_{k\in\N}\) is equicontinuous and there exists \(x_U \in U\) for which the orbit $\{F_k(x_U) \colon k\in \N\}$ has compact closure in $N$. Therefore \((F_k)_{k\in\N}\) satisfies \eqref{item:N1} and  \eqref{item:N2}, which concludes the proof.
\end{proof}


\section{Energy Gap}
\label{sec:energy-gap}

Similar to the other classes of mappings for which a version of Gromov's compactness theorem holds, quasiregular curves $M\to N$ from closed manifolds into manifolds with bounded geometry, have an energy gap, that is, a lower bound for the $L^n$-energy. 
For quasiregular curves, this follows from the H\"older continuity of the map, an injectivity radius lower bound for $N$, and de Rham's theorem for integrals of $n$-forms on $n$-manifolds. We begin with an adaptation of this last fact to our setting.

\begin{proposition}
\label{prop:trival-qrc-constant}
Let $2\le n \le m$, let $M$ be a closed, connected, and oriented, Riemannian $n$-manifold, and let $(N,\omega)$ an $n$-calibrated Riemannian $m$-manifold. Let $F\colon M \to N$ be a $K$-quasiregular $\omega$-curve. Then $F$ is constant if  $FM\subset U$ where $U\subset N$ is an open set for which $0 =[\omega|_U] \in H^n(N)$. 
\end{proposition}

\begin{proof}
Since \([\omega|_U] = 0\), the form \(\omega|_U\) is exact and there exists \(\tau \in \Omega^{n-1}(U)\) such that \(\dn\tau = \omega|_{ _U}\). Since $M$ is closed, we have, by integrating the distortion inequality and using Stokes' theorem, that
\begin{align*}
\int_{M}\norm{DF}^n &\leq K\int_{M}F^{ * }\omega =  K\int_{M}F^{ * } d\tau = K\int_{M} d F^{ * }\tau = 0.
\end{align*}
Since \(\norm{\omega}\circ F\) is non-negative, we have that \(\norm{DF} = 0\) almost everywhere. Thus by continuity \(F\) is locally constant. Since \(M\) is connected, we obtain that \(F\) is constant.
\end{proof}

For targets having bounded geometry, Proposition \ref{prop:trival-qrc-constant} yields an energy lower bound for non-constant curves; see e.g.~Parker \cite[Proposition 1.1]{Parker-JDG-1996} for an analogous statement in the context of harmonic maps. In what follows, we apply this energy gap to spherical strata of nodal manifolds.

\begin{theorem}[Energy gap]
\label{thm:analytic-energy-gap}
Let $2\le n \le m$, let $M$ be a closed, connected and  oriented Riemannian $n$-manifold, let $(N,\omega)$ be an $n$-calibrated Riemannian $m$-manifold having bounded geometry, let \(E_N > 0\) be a small energy bound for \(N\), and $K\ge 1$.  Then there exists $\varepsilon_0=\epsilon_0(N,M,K, E_N,\omega) > 0$ having the following property: Each non-constant $K$-quasiregular $\omega$-curve $F\colon M \to N$ satisfies
\[
\norm{DF}_{L^n(M)} \ge \varepsilon_0.
\]
\end{theorem}

\begin{proof}
Since $M$ is closed, we have that
\[
r_M = \inf_{p\in M} \BLrad_1(M,p)
\]
is finite and positive. We now denote $L_M = \lfloor 2 (\diam M)/r_M \rfloor +2$, where $\lfloor a \rfloor$ is the integer part of $a\in \R$.

Set 
\[
\varepsilon_0 = \frac{1}{3}\min\left\{ \frac{r_N}{C(n,\alpha) r_M^\alpha L_M }, E_N \right\}, 
\]
where $r_N = \inf_{q\in N} \BLrad_1(N,q)$ and $C(n,\alpha)$ is the constant from Corollary~\ref{cor:local-holder-cont}.

Let $F\colon M\to N$ be a $K$-quasiregular $\omega$-curve
having energy $\norm{DF}_{L^n} < \varepsilon_0$. We show that $F$ is a constant map.
By taking a geodesic between points of $M$ which realizes the diameter of $M$ and covering it with at most $L_M$ balls $B_1,\ldots, B_m$ of radius $r_M/2$ and that $\norm{DF|_{B_j}}_{L^n} < \epsilon_0 < E_N$ for each $j=1,\ldots, m$, we have, by Corollary \ref{cor:local-holder-cont}, that 
\[
\diam FM \le C(n,\alpha) \norm{DF}_{L^n} r_M^\alpha L_M < C(n,\alpha) r_M^\alpha L_M \varepsilon_0 
\le \frac{1}{3} r_N < r_N.
\]
Thus the image $FM$ of $F$ is contained in a topologically Euclidean ball $U=B_N(q,r_N)$ of $N$. Thus $0=[\omega|_U]\in H^n(U)$. Hence $F$ is constant by Proposition \ref{prop:trival-qrc-constant}.
\end{proof}


\section{Pigeonhole principle for bubbling}

In this section, we record another key element used in the proofs of Gromov's compactness theorem: \emph{the energy of a locally equibounded asymptotically quasiregular sequence concentrates on a discrete set}. This discrete set, the 'bubbling points', will be the points at which we perform nodal surgery. We record this as follows. For the statement, we denote 
\[
i_X(p) = \# \{ M\in \Strata(X) \colon x\in M\}
\]
the number of strata of the nodal manifold $X$ containing the point $p\in X$.

\begin{proposition}
\label{prop:ident-bubbling-points}
Let $X$ be a connected, oriented, and Riemannian nodal $n$-manifold, and let $N$ be a Riemannian $m$-manifold for $2\le n \le m$. Let also $(F_k \colon X \to N)_{k\in \N}$ be a locally equibounded sequence of continuous $W^{1,n}_\loc$-mappings. Then, for a constant $E>0$, there exist a discrete set $P \subset X$ and a subsequence $(F_{k_j})_{j\in \N}$ of $(F_k)_{k\in \N}$ satisfying the following conditions:
\begin{enumerate}
\item For every $x\in X\setminus P$, there exists $0<r_x<\dist_X(x,P)$ for which
\[
\sup_{j\in \N}\int_{B(x,r_x)} \norm{DF_{k_j}}^n \le i_X(p)E.
\] \label{item:ident-1}
\item For every \(p \in P\) and \(0 < r_p < \dist(p,P\setminus\sset{p})\), 
\[
\liminf_{j\to \infty}\int_{B(p,r_p)} \norm{DF_{k_j}}^n \ge E.
\]\label{item:ident-2}
\end{enumerate}
Moreover, if $\sup_{k\in \N} \norm{DF_k}_{L^n(X)} < \infty$, then \(P\) is finite. 
\end{proposition}

\begin{proof}
We consider first the special case that $X$ is a manifold and then the general case.

\smallskip
\noindent
{\bf Step I: special case.} Suppose that $X$ is a manifold.
Since $X$ is $\sigma$-compact, we may fix an exhaustion  $(X_i)_{i\in\N}$ of \(X\) by compact subsets, where $X_i \subset \interior X_{i+1}$ for each $i\in \N$. Then, for each \(i\), we have 
\[ 
R_i = \inf\set{\BLrad_1(X_i,p) \mid p \in X_i} > 0,
\]
and
\[
E_i := \sup_{k\in\N}\int_{X_i}\norm{DF_k} < \infty.
\]

Let \(\ell \in \N\) for which \(2^{-\ell} \leq R_i/2\) and let $E_{i,\ell} \subset X_i$ be an $2^{-\ell}$-net in $X_i$, that is, $X_i \subset \bigcup_{p\in E_{i,\ell}} B(p,2^{-\ell})$ and $d(p,p')\ge 2^{-\ell}$ for $p,p'\in E_{i,\ell}$, $p\ne p'$. Since, for each $p\in X_i$, the ball $B(p,R_i)$ is $2$-bilipschitz to the Euclidean ball $B^n(R_i)$, there exists a constant $Q=Q(n)$ having the property that, for each $p\in X_i$, $\# \left( B(p,2\cdot 2^{-\ell}) \cap E_\ell\right) \le Q$. In particular, the collection $\cB_{i,\ell} = \{ B(p,2^{\ell}) \colon p \in E_\ell\}$ has the property that each point in $X_i$ is contained in at most $Q$ elements of $\cB_{i,\ell}$.

Next, for each mapping \(F_k\) we gather together the balls of \(\cB_{i,\ell}\) where \(F_k\) has ``too high energy'', that is, we define
\[
\cB_{i,\ell,k} = \Set{B \in \cB_{i, \ell} \colon \int_B\norm{DF_k}^n \geq E}.
\]
Since \(\cB_{i,\ell}\) has finite cardinality, there is a subsequence \((\cB_{i,\ell,k_j})_{j\in \N}\) of \((\cB_{i,\ell,k})_{k\in \N}\) and a fixed subcollection \(\cB^b_{i,\ell} \subset \cB_{i,\ell}\) for which \(\cB_{i,\ell,k_j} = \cB^b_{i,\ell}\) for all \(j \in \N\). In particular, for each \(B \in \cB^b_{i,\ell}\) and $j\in \N$,
\begin{equation}
\label{eq:old:8}
\int_B\norm{DF_{k_j}}^n \geq E.
\end{equation}

Using the uniform upper bound for energy \(E_i\), and the energy lower bound \(E\) in each ball in $\cB^b_{i,\ell}$, we have the bound
\begin{equation}\label{eq:old:cB-card-bound}
\big|\cB^b_{i,\ell}\big| \leq \big|\cB_{i,\ell,k}\big| \leq \frac{Q E_i}{E} 
\end{equation}
for the cardinality of $\cB^b_{i,\ell}$; here we adopt the convention that \(1/\infty = 0\). 

After passing to a further subsequence $(F_{k_j})_{j\in \N}$ if necessary, the sets \(\bigcup \cB^b_{i,\ell}\) converge as \(\ell \to \infty\) under the Hausdorff metric, to a finite set \(P_i \subset X_i\). Thus \(P = \bigcup_{i=1}^{\infty}P_i\) is countable and the discreteness of \(P\) follows from the fact that \((X_i)_i\) is an exhaustion by compact sets and that each \(P_i\) is finite. Thus this subsequence $(F_{k_j})_{j\in \N}$ of $(F_k)_{k\in N}$ satisfies conditions \eqref{item:ident-1} and \eqref{item:ident-2}.

Regarding the last claim, we note that, if  \((F_k)_k\) has uniformly bounded energy, there exists \(\infty > E_0 \geq E_i\) for all \(i\in \N\). Since \(\abs{P} \leq \frac{ Q E_0}{E}\), inequality \eqref{eq:old:cB-card-bound} gives a bound for the cardinality of $P$. Thus $P$ is finite. 

\smallskip
\noindent
{\bf Step II: general case.} Suppose now that $X$ is a general nodal manifold. Since $X$ has countably many strata, we may enumerate the strata and apply the previous argument to each stratum separately. Now a diagonal argument yields a subsequence $(F_{k_j})_{j\in \N}$ of $(F_k)_{k\in \N}$ and a subset $P = \bigcup_{M\in \Strata(X)} P_M$, where each $P_M \subset M$ is discrete, satisfying conditions \eqref{item:ident-1} and \eqref{item:ident-2} in each stratum $M$ of $X$, that is,
\begin{enumerate}
\item For every $x\in M\setminus P_M$, there exists $0<r_{x,M}<{\dist}_M(x,P_M)$ for which
\[
\sup_{j\in \N}\int_{B_M(x,r_{x,M})} \norm{DF_{k_j}}^n \le E.
\] 
\item For every \(p \in P_M\) and \(0 < r_{p,M} < {\dist}_M(p,P_M\setminus\sset{p})\), 
\[
\liminf_{j\to \infty}\int_{B_M(p,r_p)} \norm{DF_{k_j}}^n \ge E.
\]
\end{enumerate}

Since $\Sing(X)$ is a discrete set and each singular point $p\in \Sing(X)$ meets only finitely many strata, we conclude that $P \subset X$ is a discrete set. For each $p\in P$, we also have that, for $x\in X\setminus P$,
\[r_x = \min\{ r_{x,M} \colon p \in M \in \Strata(M)\} >0\]
and, for $p\in P$, 
\[r_p=\min\{ r_{p,M} \colon p\in M \in \Strata(M)\}>0.\] 
Since each point $x\in X$ belongs to $i(x)$ strata, we have that the subsequence $(F_{k_k})_{j\in \N}$ satisfies conditions \eqref{item:ident-1} and \eqref{item:ident-2} in the statement. The last claim follows analogously to the case of a single stratum.
\end{proof}


\section{Singular limits of quasiregular curves}

In this section, we use the pigeonhole principle to prove a weak-$\star$ compactness and quasinormality result for sequences of mappings $(F_k \colon X\to N)_{k\in \N}$ admitting a \((K,\omega)\)-quasiregular exhaustion. For the result, we consider a class of weak-$\star$ limits of measures $\star F^*_k\omega$. These measures identify both the limiting maps and points where mass concentrates, and we use them in the proof of Theorem \ref{thm:main-reconstruction} to find the 'bubbling points', which are -- inductively -- turned into nodal points in the proof. The definition of these measures is embedded into the following definition of the singular limits of asymptotically \((K,\omega)\)-quasiregular sequences. 

\begin{definition}
\label{def:singular-limit}
Let \(X\) be an oriented and Riemannian nodal $n$-manifold, and let $(N,\omega)$ be an $n$-calibrated Riemannian $m$-manifold.
A quasiregular $\omega$-curve $F\colon X\to N$ 
is the \emph{$\delta$-singular limit of an asymptotically $(K,\omega)$-quasiregular sequence $(F_k \colon X\to N)_{k\in \N}$ for $\delta>0$} if $F_k|_{X\setminus P} \to F|_{X\setminus P}$ locally uniformly, where $P\subset X$ is a discrete set, and there exists a Radon measure $\mu$ for which
\begin{enumerate}
\item $\star F^*_k \omega \overset{*}\weakto \mu$, 
\item $\mu \llcorner (X\setminus P) = \star F^*\mu$, and 
\item $\mu(\{p\})\ge \delta$ for each $p\in P$. 
\end{enumerate}
We denote $F = \slim_{k\to \infty} F_k$, $\mu = \mlim^\omega_{k\to \infty} F_k$, and $P=\Sing(\mu)$. We also say that $(F_k)_{k\in \N}$ \emph{converges $\delta$-singularly to $F$}.
A sequence $(F_k)_{k\in \N}$ \emph{converges singularly to $F$} if $(F_k)_{k\in \N}$ converges $\delta$-singularly to $F$ for some $\delta>0$.
\end{definition}

An equicontinuous sequence of continuous Sobolev maps, admitting a quasiregular exhaustion, has a converging subsequence if the conditions of the Arzel\`a--Ascoli theorem (here Proposition \ref{prop:norm-fams-hold-qrcs}) are satisfied.

We are now ready to state a quasinormality criterion for a sequence of Sobolev maps admitting a quasiregular exhaustion. Here normality refers to having a singularly converging asymptotically quasiregular subsequence.
\begin{proposition}
\label{prop:concentration-of-mass}
Let $2\le n \le m$, let $X$ be a connected, oriented, and Riemannian nodal $n$-manifold, let $(N,\omega)$ be an $n$-calibrated Riemannian $m$-manifold having bounded geometry, and let \(E > 0\).  
Let $(F_k \colon X\to N)_{k\in \N}$ be a sequence of continuous $W^{1,n}_\loc$-Sobolev mappings having a $(K,\omega)$-quasiregular exhaustion about a discrete set $Q\subset X$ and 
satisfying the following conditions:
\begin{enumerate}
\item the family $(F_k)_{k \in \N}$ is locally equibounded; \label{item:c-o-m-0} 
\item for each $p\in Q$, the sequence \(\big(F_k|_{B(p,r_p)}\big)_{k\in \N}\) converges locally uniformly for some \(r_p > 0\); and \label{item:c-o-m-1}
\item each stratum $M$ of $X$ contains a non-empty open set $U_M \subset M$ having the property that, for each $p\in U_M$, the orbit $\{ F_k(p) \colon k\in \N \}$ has compact closure in $N$. \label{item:c-o-m-2}
\end{enumerate}
Then there exists a subsequence $(F_{k_j} \colon X\to N)_{j\in \N}$ of $(F_k)_{k\in \N}$, which is asymptotically $(K,\omega)$-quasiregular and which
converges $(E/K)$-singularly to a $K$-quasiregular curve $F \colon X\to N$.
\end{proposition}

\begin{proof}
By Proposition~\ref{prop:ident-bubbling-points}, there exist a discrete set \(P \subset X\) and a subsequence \((F_{k_j})_j\) of \((F_k)_k\) having the property that, for every \(x \in X\setminus P\), there exists \(0 < r_x < \dist(x,P)\) satisfying
\begin{equation}
\label{eq:local-energy-bound}
\sup_{j \in \N}\int_{B(x,r_x)}\norm{DF_{k_j}}^n \le i_X(x)E.
\end{equation}
Since $(F_k)_{k\in \N}$ admits a $(K,\omega)$-quasiregular exhaustion about $Q$, we have by higher integrability of quasiregular curves (see e.g.~\cite{Onninen-Pankka-qrc-higher-integrability}), the local energy bound \eqref{eq:local-energy-bound}, and by Morrey's theorem (see e.g~\cite[VII.3.1]{Rickman-book-1993}), that the subsequence $(F_{k_j}|_{X \setminus (P\cup Q)})_j$ is uniformly locally H\"older continuous, and hence that $(F_{k_j})_j$ is locally equicontinuous in $X\setminus (P\cup Q)$.

By \eqref{item:c-o-m-2} and Proposition~\ref{prop:norm-fams-hold-qrcs}, the sequence \((F_{k_j})_{j}\) has a further subsequence, denoted \((F_{k_\ell})_{\ell\in \N}\), which is asymptotically \((K,\omega)\)-quasiregular about $P\cup Q$. By Lemma \ref{lem:loc-unif-limit-energy-bound}, there exists a \(K\)-quasiregular \(\omega\)-curve \(\widetilde{F}\colon X\setminus (P\cup Q) \to N\) to which \((F_{k_\ell}|_{X\setminus (P\cup Q)})_{\ell\in \N}\) converges locally uniformly. 
Now, by \eqref{item:c-o-m-1}, we have that the sequence $(F_{k_\ell}|_{X\setminus P})_{\ell\in \N}$ converges locally uniformly to a continuous map $\widetilde F \colon X\setminus P \to N$ which is a $K$-quasiregular $\omega$-curve in $X\setminus (P\cup Q)$. By Theorem \ref{thm:cont-removability}, $\widetilde F$ is a $K$-quasiregular $\omega$-curve. 

Since the subsequence \((F_{k_\ell})_\ell\) is locally equibounded, Lemma \ref{lem:loc-unif-limit-energy-bound} additionally implies that  
\[
\int_{A}\norm{D \widetilde F}^n \leq \liminf_{\ell\to \infty}\int_{A}\norm{D F_{k_\ell}}^n \leq \sup_{\ell\in \N}\int_{A}\norm{D F_{k_\ell}}^n < \infty
\]
for each compact subset \(A \subset X\). 

Now by Theorem~\ref{thm:Ikonen-removability}, \(\widetilde F\) extends over \(P\) to a \(K\)-quasiregular \(\omega\)-curve \(F \colon X \to N\). By construction, $(F_{k_\ell}|_{X\setminus P})_{\ell\in \N}$ converges locally uniformly to $F|_{X\setminus P}$. Thus the measures $\star(F_{k_\ell}|_{X\setminus P})^*\omega$ converge in the weak-$\star$ sense to $\star(F|_{X\setminus P})^*\omega$; see \cite[Lemma 4.3]{pankka2020qrcs}. 

Since the sequence $(F_{k_\ell})_{\ell}$ is locally equibounded, we have, by Theorem~\ref{thm:weak-compact-radon-measures}, after passing to a subsequence if necessary, that measures $\star F^*_{k_\ell}\omega$ converge in the weak-$\star$ sense to a Radon measure $\mu$ on $X$. By the uniqueness of weak-$\star$ limits, $\mu\mathbin\llcorner (X\setminus P) = \star F^*\omega$.

Since \(P\) is discrete, \(\mu \mathbin\llcorner P = \sum_{p\in P} \lambda_p\delta_p\) for some \(\lambda_p \geq 0\). For each \(p \in P\), Proposition~\ref{prop:ident-bubbling-points} implies that, for sufficiently small $r_p>0$,
 \[
K\liminf_{\ell\to\infty}\int_{B(p,r_p)}F_{k_\ell}^{*}\omega \geq \liminf_{\ell\to\infty}\int_{B(p,r_p)}\norm{DF_{k_\ell}}^n \geq E.
\]
Thus \(\lambda_p \geq E/K\), and hence $(F_{k_\ell})_{\ell\in \N}$ is an asymptotically \((K,\omega)\)-quasiregular sequence which converges $(E/K)$-singularly to $F$. This concludes the proof.
\end{proof}

\begin{remark}
In \cite[Section 1.5]{Gromov-Invent-1985} Gromov calls the limit mapping $F$ a \emph{weak limit of the sequence $(F_k)$}. We have reserved this terminology for measures.
\end{remark}

\section{Nodal resolutions of singular limits}
\label{sec:first-surgery}

In this section, we discuss a method to resolve the point-masses arising in the singular limits. This formalizes the heuristic idea of bubble creation. For the discussion, we give the following definition. Note that in the definition we anticipate an iteration and hence all the bubbles in the bubble tree $\widehat X$ need not meet the original nodal manifold $X$.

\begin{definition}\label{def:nodal-resolution}
A singularly converging sequence $(\widehat F_j \colon \widehat X \to N)_{j\in \N}$ is a \emph{nodal pre-resolution} of a singularly converging sequence $(F_k \colon X\to N)_{k\in \N}$ if
\begin{enumerate}
\item $\widehat X$ is a bubble tree over $X$, \label{item:resolution-1}
\item $\slim_{j\to \infty} \widehat F_j|_X = \slim_{k \to \infty} F_k$, \label{item:resolution-2}
\item $(\pi_{\widehat X, X})_* \left( \mlim^\omega_{j\to \infty} \widehat F_j \right) = \mlim^\omega_{k \to \infty} F_k$, and \label{item:resolution-3}
\item $\Sing(\mlim^\omega_{j \to \infty} \widehat F_j)\cap X = \emptyset$. \label{item:resolution-4}
\end{enumerate}
We call \((\widehat F_j)_{j\in\N}\) a \emph{nodal resolution} of \((F_k)_{k\in\N}\) when \(\Sing(\mlim^\omega_{j \to \infty} \widehat F_j) = \emptyset\).
\end{definition}

The reader may wonder at the use of different symbols for the running indices of the sequences $(F_k)_{k\in \N}$ and $(\widehat F_j)_{j\in \N}$. In principle, these two sequences have no pairwise relation, with only the limits agreeing in the sense of the definition. That being said, we will construct nodal pre-resolutions via (iterated) nodal surgeries of a subsequence $(F_{k_j})_{j\in \N}$ of $(F_k)_{k\in \N}$, from which we obtain much stronger, though unneeded, pairwise relations.

The purpose of this definition is to formalize how the limits $\slim_{k \to \infty} F_k$ and $\mlim^\omega_{k\to \infty} F_k$ are lifted to the bubble tree $\widehat X$ over $X$. In particular, by \eqref{item:resolution-3} and \eqref{item:resolution-4}, the singularities of $\mlim^\omega_{k\to \infty} F_k$ are pre-resolved, though not yet removed, in the nodal pre-resolution in the sense that $(\mlim^\omega_{j \to \infty} \widehat F_j)(p) = 0$ and
\begin{equation}
\label{eq:resolution}
(\mlim^\omega_{j \to \infty} \widehat F_j)(\pi_{\widehat X,X}^{-1}(p)) = (\mlim^\omega_{k \to \infty}F_k)(p)
\end{equation}
for each \(p \in \Sing(\mlim^\omega_{k \to \infty}F_k)\). 
Additionally, this means that there exists a neighborhood \(V \subset \widehat X\) of \(\Sing(\mlim^\omega_{k \to \infty}F_k)\) for which \((\widehat F_j|_V)_{j\in\N}\) converges locally uniformly.

\begin{remark}
For the coming discussion, we note that a nodal pre-resolution \((\widehat F_j)_{j\in\N}\) converges locally uniformly if and only if \(\Sing(\mlim^\omega_{j \to \infty} \widehat F_j) = \emptyset\) i.e. \((\widehat F_j)_{j\in\N}\) is a nodal resolution of \((F_j)_{j\in\N}\). This is the reason for the nodal pre-resolution terminology.
Note also that nodal pre-resolutions are transitive i.e.~if \((\widetilde F_\ell \colon \widetilde X \to N)_{\ell \in\N}\) is a nodal pre-resolution of \((\widehat F_j \colon \widehat X \to N)_{j\in\N}\) and \((\widehat F_j)_{j\in\N}\) is a nodal pre-resolution of \((F_k \colon X \to N)_{k\in\N}\), then \((\widetilde F_\ell)_{\ell \in \N}\) is a nodal pre-resolution of \((F_k)_{k\in\N}\).       
\end{remark}

\subsection{Nodal Surgery of quasiregular curves}
\label{sec:nodal-surgery}

We will now define the nodal surgery of quasiregular curves and hence of asymptotically $(K,\omega)$-quasiregular sequences $(F_k \colon X\to N)_{k\in \N}$. Nodal surgery is essentially a type of extension of these classes of mappings over bubbles, which is sufficiently explicit that their constructing requires only a filling lemma; see Lemma~\ref{lemma:surgery-lemma-new}. 
Our method of constructing nodal surgeries yields a nodal manifold $X^{\vee P}$ -- a bubble tree over $X$ -- associated to a family $P$ of discrete subsets of $X$, and a sequence of asymptotically quasiregular curves $(\widehat F_k \colon X^{\vee P} \to N)_{k\in \N}$. We define the nodal manifold $X^{\vee P}$ in two steps.

For an $n$-manifold $M$ and a discrete set $P\subset M$, we set $M^{\vee P}$ to be the nodal manifold with \(\Strata(M^{\vee P}) = \{M\}\cup\set{\bS^n_p \mid p \in P}\),
where each \(\bS^n_p = \bS^n \times \{p\}\) is an isometric copy of \(\bS^n\),
i.e.
\[
M^{\vee P} = M \bigvee_{p\in P} \bS^n_p,
\]
where each $p\in P \subset M$ is identified with $(-e_{n+1},p) \in \bS^n_p$. As usual, we consider $M$ as a subset of $M^{\vee P}$.

Next, let $X$ be a $n$-nodal manifold and let $P = (P_M)_{M\in \Strata(X)}$ be a family of discrete sets, where $P_M \subset M$ for each $M\in \Strata(X)$. We denote $M_X = \coprod_{M\in \Strata(X)} M$ and let $\pi \colon M_X \to X$ be the associated quotient map. 
Since $M_X \subset \coprod_{M\in \Strata(X)} M^{\vee P_M}$, we may define $X^{\vee P}$ to be the nodal manifold
\begin{equation}
\label{eq:general-XveeP}
X^{\vee P} = \left( \coprod_{M\in \Strata(X)} M^{\vee P_M}\right)\Big/{\sim},
\end{equation}
where $\sim$ is the equivalence relation reidentifying the singular points of $X$, that is, the equivalence relation induced by the quotient map $\pi \colon M_X \to X$. Now $X \subset X^{\vee P}$ and \(\Strata_X(X^{\vee P}) = \Set{\bS^n_p \mid p \in P_M \ \text{and}\ M \in \Strata(X)}\).

Observe that, if $P_M \subset M\setminus \Sing(X)$ for each stratum $M$, we may take $P$ to be the subset $P= \bigcup_{M\in \Strata(X)} P_M$ and take
\[
X^{\vee P} = X \bigvee_{p \in P} \bS^n_p.
\]
as in the case of an $n$-manifold.

Having discussed this preliminary terminology for a ``single layer'' bubble tree over a nodal manifold, we are now ready to define the nodal surgery of a quasiregular curve. For the nodal surgery of a map $X\to N$ at a point $p\in X$ , we introduce quasiconformal embeddings $\varphi^+_{p,M} \colon B_M(p,R) \to \bS^n$ and $\varphi^-_{p,M} \colon B_M(p,R) \to \bS^n$ for $p\in M\in \Strata(X)$ as follows. For the discussion, let $\sigma \colon \R^n \to \bS^n$ be the stereographic projection for which $\sigma \R^n = \bS^n\setminus \{-e_{n+1}\}$ and let $\rho \colon \bS^n \to \bS^n$, $(x_1,\ldots, x_n, x_{n+1}) \mapsto (x_1,\ldots, x_n, -x_{n+1})$. 

Let $M$ be a Riemannian $n$-manifold, $p\in M$, and $R < \min\{1, \injrad_M(p)\}$. We denote $\xi_p = \left( \exp_p|_{B(0,R)}\right)^{-1} \colon B_M(p,R) \to \R^n$ a local inverse of the exponential map at $p$, where we identify $T_p M$ isometrically with $\R^n$. We may now fix constants $\lambda^+_R>0$ and $\lambda^-_R>0$ satisfying 
\[
\sigma(B^n(\lambda^+_R R)) = \bS^n\setminus \bar B_{\bS^n}(-e_{n+1},R) 
\]
and
\[
(\rho \circ \sigma)(B^n(\lambda^-_R R)) = B_{\bS^n}(-e_{n+1},R),
\]
respectively. We define the mappings
\[
\varphi^+_{M,p,R} \colon B_M(p,R) \to \bS^n, \quad x\mapsto \sigma( \lambda^+_R \xi_p(x)),
\]
and 
\[
\varphi^-_{M,p,R} \colon B_M(p,R) \to \bS^n, \quad x\mapsto (\rho \circ \sigma)(\lambda^-_R \xi_p(x)).
\]
Note that, by the radial symmetry of the stereographic projection, we have  
\[
\varphi^+_{M,p,R}|_{\partial B_M(p,R)} = \varphi^-_{M,p,R}|_{\partial B_M(p,R)}.
\]
We denote by $\rho_R \in (0,R)$ the radius for which 
\[
\varphi^+_{M,p,R}(B_M(p,\rho_R)) = \bS^n_p \cap \left( \R^n\times (0,\infty) \right).
\]

Having mappings $\varphi^+_{p,M}$ and $\varphi^-_{p,M}$ at our disposal, we define a nodal surgery  of a map $X\to N$ as follows.

\begin{definition}
A mapping $\widehat F \colon X^{\vee P} \to N$ is a \emph{nodal surgery of a mapping $F\colon X\to N$ over a family of discrete sets $P = (P_M)_{M \in \Strata(X)}$}, where \(P_M \subset M\) for each \(M \in \Strata(X)\), 
if there exists a family 
\[
\cB = \{ B_M(p,R_p) \colon p\in P_M,\ M \in \Strata(X)\}
\]
of mutually disjoint open balls having the properties that  
\[
\widehat F|_{X \setminus \bigcup \cB} = F|_{X \setminus \bigcup \cB}
\]
and, for each $p\in P_M$ and \(M \in \Strata(X)\), there exists $0<r_p< R_p$ for which
\begin{enumerate}
   \item $\widehat F |_{\bS^n_p\setminus \bar B_{\bS^n_p}(-e_{n+1},r_p)} \circ \varphi^+_{M,p,r_p} = F|_{B_{M}(p,r_p)}$, and
   \item $\widehat F |_{B_{\bS^n_p}(-e_{n+1},r_p)} \circ \varphi^-_{M,p,r_p} = \widehat F|_{B_{M}(p,r_p)}$.
\end{enumerate}
\end{definition}

Our basic surgery lemma about the existence of nodal surgeries of quasiregular curves reads as follows. For the statement, we denote $A_M(p;r,R) = B_M(p,R)\setminus \bar B_M(p,r)$ for $0<r<R$.

\begin{lemma}
\label{lemma:surgery-lemma-new}
Let $p\in M$ and $\varepsilon>0$. Let also $0<r'<r < R< \QCrad_{\varepsilon}(M,p)$ and let $F_p \colon B_M(p,R) \to N$ be a $K$-quasiregular $\omega$-curve. Let also $X_p = B_M(p,R)\vee \bS^n_p$.
If $F_pA_M(p;r',r) \subset B_N(q,\widehat R)$, where $\widehat R < \BLrad_1(N,q)$ for some $q\in N$,
then there exists a Sobolev map $\widehat F_p \colon X_p \to N$ in $W^{1,n}_\loc(X_p,N)$ having the following properties:
\begin{enumerate}
\item the restriction $\widehat F_p|_{X_p\setminus B_{X_p}(p,r)}$ is a $(1+\varepsilon)K$-quasiregular $\omega$-curve; \label{item:sl-1}
\item $\widehat F_p$ is a nodal surgery of $F_p$ over $\{p\}$; \label{item:sl-2}
\item $\norm{D\widehat F_p}_{L^n(X_p)} \le C(n)(1+\varepsilon)^{1/n} \left( \norm{DF}_{L^n(A_M(p;r',R))} +\left( \log \frac{r}{r'}\right)^\frac{1-n}{n} \right)$; and  \label{item:sl-3}
\item $\widehat F_p(B_{X_p}(p,r)) \subset B_N(q, \widehat R)$. \label{item:sl-4}
\end{enumerate}
\end{lemma}

\newcommand{\cc}{\mathrm{cap}}
\newcommand{\cut}{\mathrm{cut}}

\begin{proof}
We begin by defining an auxiliary map $\widehat f \colon B_M(p,R) \to N$.
For this, let $\varphi \colon B_N(q,\widehat R) \to B^m(\widehat R)$ be a chart at $q$ given by the exponential map of $N$ at $q$ and let $\psi \colon B_M(p,R)\to N$ be a chart at $p$ given by the exponential map of $M$ at $p$. Also let $\widetilde F \colon A(r',r) \to B^n(\widehat R)$ be the map $x\mapsto \varphi \circ F \circ \psi^{-1}(x)$, where $A(r',r) = \bar B^n(r)\setminus B^n(r')$. 

We fix a smooth function $u\colon B^n(R) \to [0,1]$, satisfying $u(x) = 0$ for $|x|\le r'$ and $u(x)=1$ for $|x|\ge r$,  which satisfies
\[
\norm{\grad u}_{L^n(A(r'r)} \le 2 \cc_n(\bar B^n(r'), B^n(r)) = 2 c(n)\log \left( \frac{r}{r'} \right)^{1-n};
\]
see e.g.~\cite[Section 2, Example 2.12]{heinonen2018nonlinear} for a discussion of conformal capacity.

We also define a map $\widetilde f \colon B^n(r) \to B^m(\widehat R)$ by the formula 
\[
x\mapsto \left\{\begin{array}{ll}
u(x)\widetilde F(x), & \text{if } |x|\ge r'\\
0, & \text{if }|x|\le r'.
\end{array}\right.
\]
Next, we define another map $\widehat f \colon B_M(p,R) \to N$ by the formula
\[
\widehat f(x) = \left\{ \begin{array}{ll}
F(x), & \text{if } x\not \in B_M(p,r) \\
\varphi^{-1}\left(\widetilde f(\psi(x))\right), & \text{if } x\in B_M(p,r).
\end{array}\right.
\]
Then $\widehat f(B_M(p,R)) \subset B_N(q,\widehat R)$, and clearly \(\widehat f \in W^{1,n}_\loc(B_M(p,R), N)\). Before completing the construction of the mapping $\widehat F_p$, we show that 
\begin{equation}
\label{eq:energy-mollification}
\norm{D\widetilde f}_{L^n(B^n(r))} \le (1+\varepsilon)^{1/n} \left( 3\norm{D\widetilde F}_{L^n(A_M(p;r',R))} + c(n)\left( \log \frac{r}{r'}\right)^\frac{1-n}{n} \right).
\end{equation}
We first observe that
\[
D\widetilde f(x) = u(x)D\widetilde F(x) - \langle \grad u(x), \widetilde F(x) \rangle
\]
for almost every $x\in B^n$. Since $u(x)=0$ for $x\in B^n(r')$ and $\grad u(x) = 0$ for $x\not \in A(r',r)$, we have that
\begin{align*}
|D\widetilde f(x)| &\le|D\widetilde F(x)| + |\grad u(x)|
= \chi_{A(r',R)}(x)|D\widetilde F(x)|+ \chi_{A(r',r)}(x)|\grad u(x)|
\end{align*}
for almost every $x\in B^n(R)$, where $\chi_E$ is the characteristic function of a set $E \subset B^n(R)$. Thus, by Minkowski's inequality,
\begin{align*}
\norm{D\widetilde f}_{L^n(B^n(R))} &= \norm{\chi_{A(r',R)}|D\widetilde F|+ \chi_{A(r',r)}(x)|\grad u|}_{L^n(B^n(R))} \\
&\le \norm{\chi_{A(r_2,R)}|D\widetilde F|}_{L^n(B^n(R))}+ \norm{\chi_{A(r',r)}(x)|\grad u|}_{L^n(B^n(R))} \\
&=  \norm{D\widetilde F}_{L^n(A(r',R))}+ \norm{\grad u}_{L^n(A(r',r))} \\
&\le \norm{D\widetilde F}_{L^n(A(r',R))}+ 2 c(n) \log \left( \frac{r}{r'}\right)^{1-n}.
\end{align*}

We are now ready to define the mapping $\widehat F_p$. 
Set $\widehat F_p = \widehat f$ on $B_M(p,R)$, and on the sphere $\bS^n_p$ set 
\[
\widehat F_p|_{\bS^n\setminus B_{\bS^n}(-e_{n+1}, r)} = F \circ (\varphi^+_{M,p,R})^{-1} 
\text{ and } 
\widehat F_p|_{B_{\bS^n}(-e_{n+1},r)} = \widehat f \circ (\varphi^-_{M,p,R})^{-1}.
\]
Since $R<\QCrad_\varepsilon(M,p)$, the mappings $\varphi^+_{M,p,R}$ and $\varphi^-_{M,p,R}$ are $(1+\varepsilon)$-quasiconformal. Thus
we have that $\widehat F_p$ is in \(W^{1,n}_\loc(X_p,N)\), is a $(1+\varepsilon)K$-quasiregular $\omega$-curve on $X_p \setminus B_{X_p}(p,r)$, and that \eqref{item:sl-1} holds.
Furthermore, $\widehat F_p$ is a nodal surgery of $F_p$ by definition. Hence \eqref{item:sl-2} is satisfied. Finally \eqref{item:sl-3} holds by \eqref{eq:energy-mollification} and change of variables, and \eqref{item:sl-4} holds by the choice of $\widehat f$.
\end{proof}

\subsection{Construction of nodal pre-resolutions via surgery}

We are now ready to show how to construct a nodal pre-resolution of a singularly converging asymptotically quasiregular sequence using nodal surgery.

\begin{proposition} 
\label{prop:first-surgery-nodal-manifold}
Let $2\le n \le m$, let $X$ a connected, oriented, and Riemannian nodal $n$-manifold, and let $(N, \omega)$ be an \(n\)-calibrated Riemannian \(m\)-manifold having bounded geometry.
Let $(F_k \colon X\to N)_{k\in \N}$ be a locally equibounded and singularly converging asymptotically $(K,\omega)$-quasiregular sequence for $K\ge 1$. 
Then, for
\[
P = \left(\Sing(\mlim^\omega_{k \to \infty} F_k|_{M})\right)_{M \in \Strata(X)},
\] 
there exists a locally equibounded and asymptotically $(K,\omega)$-quasiregular sequence 
\[
(\widehat F_j \colon X^{\vee P} \to N)_{j\in \N},
\]
which is a nodal pre-resolution of a subsequence \((F_{k_j})_{j\in\N}\) of $(F_k)_{k \in \N}$.
\end{proposition}

\begin{remark}
The mapping \(\widehat F_j\) in the nodal pre-resolution \((\widehat F_j)_{j\in\N}\) is obtained as a nodal surgery over \(P\) of the mapping $F_{k_j}$. 
Additionally, by construction the sequence $( \widehat F_j)_{j\in \N}$ has the property that its singular sets $\Sing(\mlim^\omega_{j \to \infty} \widehat F_j)$ are contained in the upper hemispheres of the bubbles in $\Strata_X(X^{\vee P})$.
\end{remark}

\begin{proof}[Proof of Proposition \ref{prop:first-surgery-nodal-manifold}]
We prove the claim first in the special case of a manifold; this special case is applied in the second step of the proof to the strata of the nodal manifold $X$.

\bigskip
\noindent
{\bf Step I: Special case.} Let $M$ be a connected, oriented, and Riemannian $n$-manifold.
We begin by fixing some notation for the proof. Let $F= \slim_{k \to \infty} F_k$ and $\mu = \mlim^\omega_{k \to \infty} F_k$. Set $P=\Sing(\mu)$ and let $\mu_k = \star F_k^*\omega$ for each $k\in \N$. 
For each $p\in P$, we fix a radius $R_p>0$ for which the balls $B_M(p,R_p)$ are mutually disjoint.
We also fix, for each $p\in P$ and $m\in \N$, 
\[
R_{p,m} = (1/2) \min\{ \QCrad_{1/m}(M,p), R_p/m\}.
\]
For completeness, we set $R_{p,0}=R_p$.

For each $p\in P$ and $m\in \N$, we also set $r_{p,m} = R_{p,m}/m$ and let $r'_{p,m} = \rho_{r_{p,m}}< r_{p,m}$, that is, the radius for which $\varphi^+_{p,r_{p,m}}(B_M(p,r'_{p,m}))$ is the upper hemisphere of $\bS^n_p$. For completeness, we also define $r_{p,0}=R_p/2$, and $r'_{p,0}= R_p/4$. For these radii, we define  
\[
B_{p,m} = B_M(p,R_{p,m}) \text{ and } A'_{p,m} = A_M(p; r'_{p,m}, R_{p,m}). 
\]

Let $p\in P$. We fix a sequence $(k^p_m)_{m\in \N}$ of indices associated to the concentration of mass at $p$ as follows. Recall that, since $F_k$ converges singularly to $F$, we have that $\mu \llcorner B_M(p,R_p) = (\star F^*\omega)\llcorner B_M(p,R_p)$$ + \mu(\{p\}) \delta_p$, where $\delta_p$ is the Dirac mass at $p$. Since $\mu(\{p\})>0$ and $(\star F^*\omega)(B_M(p,r)) \to 0$ as $r\to 0$, we have that
\[
\frac{\mu(B_M(p,r)\setminus \{p\})}{\mu(B_M(p,r))} \to 0
\]
as $r\to 0$. Furthermore, by the continuity of $F$, $\diam(F(B_M(p,r)) \to 0$ as $r \to 0$.

Since $\mu_k \weakto \mu$ in $B_M(p,R_p)$, there exists an increasing sequence $(k^p_m)_{m\in \N}$ of indices, where $k^p_0=0$, having the property that, for each $m\ge 1$, 
\[
\frac{\mu_k(A'_{p,m})}{\mu_k(B_{p,m})} \le \frac{1}{m}
\quad \text{and} \quad 
F_k(A'_{p,m}) \subset B_N\left(F(p), \frac{\BLrad_1(N)}{2m}\right)
\]
for $k\ge k^p_m$. Next we define $(m^p_k)_{k\in \N}$ to be the sequence for which $m^p_k$ is the largest index $m$ satisfying $k^p_m \le k < k^p_{m+1}$. Clearly, $m^p_k \to \infty$ as $k\to \infty$ for each $p\in P$. We also denote $\varepsilon_k = 1/m^p_k$ for each $k\in \N$.

\bigskip

We define a sequence $(F'_k \colon M^{\vee P} \to N)_{k\in \N}$ by taking local nodal surgeries of maps $F_k$ at points of $P$ as follows. Let $k\in \N$ and $p\in P$. We set the restriction $F'_k|_{B_{p,m^p_k}} \colon B_{p,m^p_k} \to N$ to be the nodal surgery of $F_k|_{B_{p,m^p_k}}$ from Lemma \ref{lemma:surgery-lemma-new} given by the radii $r'= r'_{p,m^p_k}$, $r = r_{p,m^p_k}$, $R=R_p$, $\widehat R = \BLrad_1(N)/(2m^p_k)$, and $q=F(p)$. Now let $\Omega_k =\bigcup_{p\in P} (B_M(p,R_p) \cup \bS^n_p) \subset M^{\vee P}$ be the open set, where $\widehat F_k$ is already defined. Then $M^{\vee P} \setminus \Omega_k = M\setminus \Omega_k$. Since $F'_k|_{\Omega_k \cap M}$ agrees with $F_k|_{\Omega_k \cap M}$ in a neighborhood of $\partial \Omega_k$ in $\Omega_k \cap M$, we obtain a continuous Sobolev map in $W^{1,n}_\loc(M^{\vee P},N)$ by defining $F'_k|_{M^{\vee P}\setminus \Omega_k} = F_k|_{M^{\vee P}\setminus \Omega_k}$.

To analyze the nodal surgeries at individual points $p\in P$, we observe first that, by the nodal surgery given by Lemma~\ref{lemma:surgery-lemma-new} and a choice of the radius $r'_{p,m^p_k}$, we have the estimate
\[
\norm{DF'_k}_{L^n(X_{p,k})} \le C(n) (1+\varepsilon_k)^{1/n} \left( \norm{DF_k}_{L^n(A'_{p,m_k})} 
+ \log\left( \frac{r_{p,m^p_k}}{r'_{p,m^p_k}}\right)^{\frac{1-n}{n}} \right) \to 0 
\]
as $k\to \infty$, where \(X_{p,k} = B_M(p,R_{m^p_k}) \vee \Sbb^n_p\). Thus the sequence $(F'_k \colon M^{\vee P}\to N)_{k\in \N}$ is locally equibounded. Additionally, by \eqref{item:sl-4} in Lemma~\ref{lemma:surgery-lemma-new}, we have that 
\[
F'_k(B_{X_{p,k}}(p, r_{p, m_k^p})) \subset B_N(F(p), \BLrad_1(N)/(2m^p_k))
\]
for each \(k \in \N\) and \(p \in P\). Since \(\BLrad_1(N)/(2m^p_k)\to 0\) as \(k \to \infty\) and \((F_k)_{k \in \N}\) converges singularly to \(F\), we obtain that \((F'_k)_{k \in \N}\) converges locally uniformly at \(p\). In particular, the sequence $(F'_k)_{k\in\N}$ does not concentrate mass at $p$.

By the nodal surgery, we also have that the sequence $(F'_k|_{\bS^n_p})_{k\in \N}$ converges in a neighborhood of $p=-e_{n-1}$ locally uniformly, and for any $\varepsilon>0$, and any neighborhood $V_p$ of $p$ in $\bS^n_p$, there exists $k_p\in \N$ for which the maps $F'_k|_{\bS^n_p\setminus \overline{V_p}}$ are $(1+\varepsilon)K$-quasiregular $\omega$-curves for $k\ge k_p$. Since the restrictions $F'_k|_M \colon M\to N$ converge locally uniformly to $F=\slim_{k \to \infty} F_k$, we conclude that, for each relatively compact open set $\Omega \subset X$ and $\widehat \Omega = \pi_{M^{\vee P},M}^{-1}(\Omega)$, the sequence of restrictions $(F'_k|_{\widehat \Omega} \colon \widehat \Omega \to N)_{k\in \N}$ admits a $(K,\omega)$-quasiregular exhaustion; this follows immediately from the finiteness of the set $P\cap \Omega$ and existence of an exhaustion in each sphere $\bS^n_p$. Note that, due to nodal surgery, the sequence $(F'_k)_{k\in \N}$ converges locally uniformly in a neighborhood of \(P\).
By Proposition~\ref{prop:norm-fams-hold-qrcs}, the sequence $(F'_k)_{k\in \N}$ has compact orbits in the sense that each stratum $S$ of $\widehat X$ contains an open set $U_S \subset S$, whose orbit $\{ F'_k U_S\}_{k\in \N}$ is contained in a compact set.

We fix now an exhaustion $U_1 \subset \overline{U_1} \subset U_2 \subset \cdots \subset X$ by relatively compact open sets and denote $\widehat U_\ell = \pi_{M^{\vee P},M}^{-1}(U_j)$ for each $j\in \N$. By Proposition \ref{prop:concentration-of-mass}, the sequence $(F'_k|_{\widehat U_\ell} \colon \widehat U_\ell \to N)_{k\in \N}$ has an asymptotically $(K,\omega)$-quasiregular subsequence for each $\ell\in \N$. Thus, by passing to a diagonal subsequence if necessary, we obtain an asymptotically $(K,\omega)$-quasiregular subsequence $(F'_{k_j} \colon M^{\vee P}\to N)_{j\in \N}$ of $(F'_k)_{k\in \N}$, which converges singularly to a $K$-quasiregular $\omega$-curve $M^{\vee P}\to N$. By the choice of the radii $r'_{p,m}$, we further have that the singular parts of the measure $\mlim^\omega_{j \to \infty} F'_{k_j}$ is contained in the union of upper hemispheres of the bubbles $\bS^n_p$ for $p\in P$. We set $\widehat F_j = F'_{k_j}$ for each $j\in\N$.

It remains to check that $(\widehat F_j \colon M^{\vee P} \to N)_{j\in \N}$ is a nodal pre-resolution of $(F_k)_{k\in \N}$. By construction $M^{\vee P}$ is a bubble tree over $M$. By nodal surgery and singular convergence, we have that $\slim_{j \to \infty} \widehat F_j|_M = \slim_{k \to \infty} F_k$ and $(\pi_{M^{\vee P},M}\mlim^\omega_{j \to \infty} \widehat F_j)(\bS^n_p) = (\mlim^\omega_{k\to \infty}F_k)(\{p\})$ for each $p\in P$. Thus $(\pi_{M^{\vee P},M})_*(\mlim^\omega_{j\to \infty} \widehat F_j) = \mlim^\omega_{k \to \infty} F_k$. Finally, since \(\Sing(\mlim^\omega_{j \to \infty} \widehat F_j)\) is contained in the upper hemispheres in the bubbles in \(\Strata_M(M^{\vee P})\), 
we have that $\Sing(\mlim^\omega_{j \to \infty} \widehat F_j) \cap M = \emptyset$.

\bigskip

\noindent {\bf Step II: General case.} We move to the second part of the proof, and consider the case where $X$ is a connected, oriented, and Riemannian nodal $n$-manifold, as in the statement. We denote $\mu = \mlim^\omega_{k\to \infty} F_k$. For each stratum $M\subset X$, we have that 
\[
\slim_{k\to \infty} F_k|_M = F|_M,
\quad
\mlim^\omega_{k\to \infty} F_k|_M = \mu \llcorner M,
\quad
\text{and}
\quad
P_M = \Sing(\mu \llcorner M).
\]

Since $X$ has countably many strata, by passing to a diagonal subsequence, we obtain a subsequence $(F_{k_j})_{j\in \N}$ of $(F_k)_{k\in \N}$ having the property that, for each $j\in \N$ and $M\in \Strata(X)$, there exists a singularly converging asymptotically \((K,\omega)\)-quasiregular sequence $(\widehat F^M_j \colon M^{\vee P_M} \to N)_{j\in\N}$ which is a nodal pre-resolution of \((F_{k_j}|_M)_{j\in\N}\) where each \(\widehat F^M_j\) is a nodal surgery of \(F_{k_j}|_M\) over \(P_M\) and for which \(\Sing (\mlim^\omega_{k\to \infty} \widehat F^M_j)\) is contained in the upper hemispheres of the bubbles in \(\Strata_M(M^{\vee P_M})\). 
This all follows from Step I of this proof. We combine these to construct a nodal pre-resolution of \((F_k)_k\).

Let $\widehat F^M \colon M^{\vee P_M} \to N$ be the singular limit $\widehat F^M = \slim_{j \to \infty} \widehat F^M_j$ of the sequence $(\widehat F^M_j)_{j\in \N}$ and let $\mu^M = \mlim^\omega_{j\to \infty} \widehat F^M_j$.
We take $X^{\vee P}$ to be the nodal manifold 
\[
X^{\vee P} = \left( \coprod_{M\in \Strata(X)} M^{\vee P_M} \right) \Big / \sim
\]
as in \eqref{eq:general-XveeP}. In what follows, we denote $\widehat X = X^{\vee P}$. Thus \(\widehat X\) is a bubble tree over \(X\) and so satisfies condition~\eqref{item:resolution-1} of Definition~\ref{def:nodal-resolution}.

Now for each \(j \in \N\) and \(M \in \Strata(X)\), set \(\widehat F_j|_{M^{\vee P_M}} = \widehat F^M_j\) and \(\widehat F|_{M^{\vee P_M}} = \widehat F^M\). 
Then these are well-defined mappings \(\widehat X \to N\) since \(\widehat F_j^M |_{M \cap \Sing(X)} = \widehat F_j|_{M \cap \Sing(X)}\), and \(\widehat F^M |_{M \cap \Sing(X)} = \widehat F|_{M \cap \Sing(X)}\) for all \(M \in \Strata(X)\).
Thus \((\widehat F_j)_j\) converges singularly to \(\widehat F\). 
In particular, for \(M \in \Strata(X)\), \(\slim_{j\to\infty} \widehat F_j^M |_{M} = \slim_{j\to\infty} F_{k_j} |_{M}\), and hence \(\slim_{j\to\infty} \widehat F_j |_{X} = \slim_{j\to\infty} F_{k_j}\).
Thus condition~\eqref{item:resolution-2} of Definition~\ref{def:nodal-resolution} is satisfied.

Observe that since each \(\widehat F^M_j\) is a nodal surgery of \(F_{k_j}|_M\) over \(P_M\) and recalling that the definition of nodal surgery is local on each stratum, we obtain that each \(\widehat F_j\) is a nodal surgery of \(F_{k_j}\) over \(P\).

Additionally since for each \(M \in \Strata\), the singularities \(\Sing (\mu^M)\) are contained in the upper hemispheres of the bubbles in \(\Strata_M(M^{\vee P_M})\), also for \(\widehat \mu = \mlim^\omega_{j\to\infty} \widehat F_j\), \(\Sing(\widehat \mu)\) is contained in the upper hemispheres of the bubbles in \(\Strata_X(\widehat X)\). Thus condition~\eqref{item:resolution-4} of Definition~\ref{def:nodal-resolution} is also satisfied.

For condition \eqref{item:resolution-3} of Definition~\ref{def:nodal-resolution}, recall that $\widehat X$ is a nodal manifold having $X$ and each $M^{\vee P_M}$ as nodal submanifolds. We also have that the nodal projection $\pi_{\widehat X,X} \colon \widehat X\to X$ is uniquely defined by the restrictions $\pi_{\widehat X, X}|_{M^{\vee P_M}} = \pi_{M^{\vee P_M}, M}$ for $M\in \Strata(X)$. 
We observe that $\widehat\mu = \mlim^\omega_{j\to \infty} \widehat F_j$ satisfies 
\begin{align*}
\widehat\mu\llcorner (\widehat X \setminus X) 
&= \sum_{S \in \Strata_{\widehat X}(X)} \widehat\mu \llcorner (S\setminus X) 
= \sum_{M\in \Strata(X)} \sum_{S \in \Strata(M^{\vee P_M})} \widehat\mu \llcorner (S\setminus M) \\
&= \sum_{M\in \Strata(X)} \sum_{S \in \Strata(M^{\vee P_M})} \mu^M \llcorner S 
\end{align*}
and
\[
\widehat\mu \llcorner X = \widehat \mu \llcorner (X\setminus \Sing(X)) 
= \sum_{M\in \Strata(X)} \widehat\mu \llcorner (M\setminus S_M) 
= \sum_{M\in \Strata(X)} \mu^M \llcorner M
\]
for each $M\in \Strata(X)$. Thus
\[
\widehat\mu  = \sum_{M\in \Strata(X)} \mu^M.
\]

Recall that for each \(M \in \Strata(X)\), \((\widehat F^M_j)_{j\in\N}\) is a nodal pre-resolution of \((F_{k_j}|_M)_{j\in\N}\) and thus satisfies
\[
(\pi_{M^{\vee P_M}, M})_*(\mu^M) = (\pi_{M^{\vee P_M}, M})_*(\mlim^\omega_{j \to \infty} \widehat F^M_j) = \mlim^\omega_{j \to \infty} F_{k_j}|_M.
\]
Thus we have have that
\begin{align*}
(\pi_{\widehat X, X})_* (\mlim^\omega_{j\to \infty} \widehat F_j) 
&= \sum_{M\in \Strata(X)} (\pi_{\widehat X,X})_*(\mu^M) \\
&= \sum_{M\in \Strata(X)} (\pi_{M^{\vee P_M}, M})_*(\mu^M) \\
&= \sum_{M\in \Strata(X)} \mlim^\omega_{j\to\infty} F_{k_j}|_M
= \mlim^\omega_{j\to\infty} F_j.
\end{align*}
Hence \((\widehat F_j)_{j \in \N}\) satisfies condition~\eqref{item:resolution-3}, and therefore is a nodal pre-resolution of \((F_{k_j})_{j \in \N}\).

Since the sequence $(\widehat F_j^M)_{j\in \N}$ is locally equibounded for each stratum \(M \in \Strata(X)\), with data independent of the strata, we conclude that also the sequence $(\widehat F_j)_{j\in \N}$ is locally equibounded.
This concludes the proof.
\end{proof}


\section{Renormalization of nodal pre-resolutions}

In the previous section, we showed that we can construct a nodal pre-resolution $(\widehat F_j \colon X^{\vee P} \to N)_{j\in \N}$ of an asymptotically \((K,\omega)\)-quasiregular sequence $(F_k \colon X \to N)_{k\in N}$ using nodal surgery. 
In this section, we show that the sequence $(\widehat F_j)_{j\in \N}$ may be renormalized i.e.~$(\widehat F_j)_{j\in \N}$ may be modified in the bubbles \(\Strata_{X}(X^{\vee P})\) in such a way that the mass of \(\mlim^\omega_{j\to\infty} \widehat F_j\) is redistributed, resulting in a nodal pre-resolution with quantitatively smaller atoms. We state this as follows.

\begin{proposition}
\label{prop:global-renormalization}
Let $2\le n \le m$, let $X$ be a connected, oriented Riemannian nodal $n$-manifold, 
and let $(N,\omega)$ be an \(n\)-calibrated Riemannian \(m\)-manifold having bounded geometry. 
Let \((\widehat F_j \colon X^{\vee P} \to N)_{j \in \N}\) be a locally equibounded and asymptotically \((K,\omega)\)-quasiregular sequence, which is a nodal pre-resolution of $(F_k \colon X \to N)_{k\in \N}$, where 
\[
P = \left(\Sing(\mlim^\omega_{k \to \infty} F_k|_{M})\right)_{M \in \Strata(X)}.
\] 
Then there exists a locally equibounded and asymptotically $(K,\omega)$-quasiregular sequence 
\[
(\widetilde F_\ell\colon X^{\vee P} \to N)_{\ell\in \N},
\]
which is a nodal pre-resolution of $(F_k \colon X \to N)_{k\in \N}$, and satisfies
\begin{equation}
\label{eq:split-of-atom}
(\mlim^\omega_{\ell\to \infty} \widetilde F_\ell)(\{ \tilde p \}) \le (9/10) (\mlim^\omega_{j \to \infty} \widehat F_j)(\bS^n_p)
\end{equation}
for each $M\in \Strata(X)$, $p\in P_M = \Sing(\mlim^\omega_{k\to\infty} F_k|_M)$, and $\tilde p\in \bS^n_p$.  
\end{proposition}

\begin{remark}
The renormalization in Proposition \ref{prop:global-renormalization} is given by a sequence $\left( h_j \colon X^{\vee P} \to X^{\vee P}\right)_{j\in \N}$ of $1$-quasiconformal mappings satisfying $h_j|_X = \id$ and we may take a subsequence of $(\widehat F_j \circ h_j)_{k\in \N}$ to be the sequence $(\widetilde F_\ell)_{\ell\in \N}$. Note that, since each $h_j$ is the identity on $X$, we have that curves $\widehat F = \slim_{j \to \infty} \widehat F_{j}$ and $\widetilde F = \slim_{\ell\to \infty} \widetilde F_{\ell}$ agree on $X$. 
\end{remark}

\subsection{Renormalization in the Euclidean space}

For the proof of Proposition~\ref{prop:global-renormalization}, we begin by recalling a redistribution result for measures in the Euclidean space discussed in 
\cite{pankka-souto2023bubbleqr}.

\newcommand{\cm}{\mathrm{cm}}

Let $\mu$ be a non-atomic measure on $\R^n$ with \(\int_{\R^n}\abs{x}\dn\mu(x) < \infty\) and whose support $\spt(\mu)$ is not contained in an affine line. The \emph{center of mass $\cm(\mu) \in \R^n$ of \(\mu\)} is the unique minimum of the function
\[
\Phi_{\mu}\colon \R^n \to \R,\ x \mapsto \int_{\R^n}\abs{x - z}\dn\mu(z).
\]
The conditions on \(\mu\) ensure that \(\Phi_{\mu}\) is well-defined and has a unique minimum. For the details, see the paragraph following \cite[Proposition~3.3]{pankka-souto2023bubbleqr}. As the name suggests, heuristically, the center mass $\cm(\mu)$ of $\mu$ is located where the measure has most mass. The following formulation of this heuristic principle suffices for us.

\begin{lemma}[{\cite[Lemma 3.4]{pankka-souto2023bubbleqr}}]
\label{lem:centre-mass-meas-ball}
Let $\mu$ be a finite non-atomic measure on $\R^n$ with \(\int_{\R^n}\abs{x}\dn\mu(x) < \infty\) and whose support is not contained in a line. If $\mu(B^n(0,1)) > \frac{3}{4} \mu(\R^n)$, then $\cm(\mu) \in B^n(0,2)$. 
\end{lemma}

A basic redistribution result for a measure concentrated near the origin reads as follows; see \cite[Proposition 3.3]{pankka-souto2023bubbleqr}. We recall a simple proof for the readers convenience. 

\begin{proposition}
\label{prop:affine-mass-redistribution}
  Let $\mu$ be a finite non-atomic measure on $B^n(0,1) \subset \R^n$, whose support is not contained in a line. 
    Suppose there exists $0 < \epsilon < \frac{1}{8}$ and \(0 < \gamma < \frac{1}{8}\norm{\mu}\) for which \(\mu(B^n(0,\epsilon)) \geq \norm{\mu} - \gamma\).
   Then for $\gamma < \delta < \frac{1}{8}\norm{\mu}$, there exists $\Lambda>\frac{1}{3\epsilon}$, a conformal map 
   \[
   A \colon \R^n \to \R^n,\ x\mapsto \Lambda(x - \cm(\mu)) + \cm(\mu)
   \]
   and \(R = 2\Lambda - 1\), which satisfy the following properties:
  \begin{enumerate}
  \item $\supp A_*\mu \subset \bar B^n(0,R)$
  \item \(A_{*}\mu (B^n(0,R)\setminus B^n(0,1+2\epsilon)) \leq 2\delta\), and
  \item \(A_{*}\mu (B^n(z,1/4)) \leq \norm{\mu} - \delta\) for all \(z \in B^n(0,R)\).
  \end{enumerate}
\end{proposition}

The proof consists of two parts: a proof for the special case $\cm(\mu) = 0$ and an affine conjugation of this result. We discuss the redistribution first as a separate lemma.
\begin{lemma}\label{lemma:scaling-mass-redistribution}
  Let $\mu$ be a finite non-atomic measure on $B^n(0,1) \subset \R^n$, whose support is not contained in a line, and with $\cm(\mu) = 0$. 
  Suppose there exists $0 < \epsilon < 1$ and 
  \(0 < \gamma < \frac{1}{8}\norm{\mu}\) for which 
   \(\mu(B^n(0,\epsilon)) \geq \norm{\mu} - \gamma\).
   Then, for $\gamma < \delta < \frac{1}{8}\norm{\mu}$, there exists $\Lambda > \frac{1}{\epsilon}$ and an affine map $T \colon \R^n \to \R^n$, $x\mapsto \Lambda x$,
   satisfying the following properties:
  \begin{enumerate}
  \item \(T_{*}\mu (B^n(0,\Lambda)\setminus B^n(0,1)) = 2\delta\) and
  \item \(T_{*}\mu (B^n(z,1/4)) \leq \norm{\mu} - \delta\) for all \(z \in \R^n\).
  \end{enumerate}
\end{lemma}

\begin{proof}
  Consider the function
  \[
  \Psi \colon [1,\infty) \to [0,\norm{\mu}),\ \lambda \mapsto \mu\big( B^n(0,1) \setminus B^n(0, 1/\lambda) \big).
  \]
  Since \(\mu\) has no atoms, \(\Psi\) is continuous. 
  Also observe that \(\Psi(1) = 0\) and \(\Psi(\lambda) \to \norm{\mu}\) as \(\lambda \to \infty\). Thus there exists \(\Lambda > 1\) for which 
  \(\Psi(\Lambda) = 2\delta > 0\).
  Let \(T \colon \R^n \to \R^n\) be the affine map 
  \(x \mapsto \Lambda x\).
Then
  \begin{equation}\label{eq:scaling-mass-redist} 
    2\delta = \Psi(\Lambda) 
    = \mu\big( B^n(0,1) \setminus B^n(0, 1/\Lambda) \big) 
    = T_{*}\mu(B^n(0,\Lambda)\setminus B^n(0,1)).
  \end{equation}
  Thus (1) holds. In addition, we have that
  \[
  2\delta = \norm{\mu} - \mu(B^n(0, 1/\Lambda)) 
  \leq \mu(B^n(0, \epsilon)) + \gamma - \mu(B^n(0, 1/\Lambda)). 
  \]
  Thus
  \[
  \mu(B^n(0,\epsilon)\setminus B^n(0,1/\Lambda)) = 2\delta - \gamma > \delta > 0,
  \]
  which implies that $\Lambda>1/\epsilon$.
  
  We prove (2) by contradiction. Suppose there exists 
  \(z \in \R^n\) for which 
  $T_*\mu(B n(z,1/4))>\norm{\mu}-\delta$. Then
  \[
  \mu(B^n(z/\Lambda,1/(4\Lambda)) 
  = T_{*}\mu(B^n(z,1/4)) 
  \geq \norm{\mu} - \delta \geq \norm{\mu} - 2\delta 
  \geq \frac{3}{4}\norm{\mu} > 0.
  \]
  Since $T$ is a scaling, measures \(T_{*}\mu\) and \(\mu\) have the same center of mass at the origin.
  Observe also that \(\norm{\mu} = \norm{T_{*}\mu}\).
  By a translation and scaling, Lemma~\ref{lem:centre-mass-meas-ball} gives that 
  \(0 \in B^n(z,1/2)\).
  From \eqref{eq:scaling-mass-redist}, we obtain 
  \(\norm{\mu} - 2\delta = \mu(B^n(0,1/\Lambda)) > 0 \)
  which implies
  \(T_{*}\mu(B^n(z,1/4)\setminus B^n(0,1)) > 0\). Hence $B^n(z,1/4) \not \subset B^n(0,1)$ implying 
  \(\abs{z} > 1/2\). Thus \(0 \not\in B^n(z,1/2)\), which is a contradiction.
\end{proof}

\begin{proof}[Proof of Proposition~\ref{prop:affine-mass-redistribution}]
  Let $x_0 = \cm(\mu)$.
  Since \(\gamma < \frac{1}{8}\norm{\mu}\), Lemma~\ref{lem:centre-mass-meas-ball} gives that \(x_0 \in B^n(0,2\epsilon)\).
  Let \(S\colon \R^n\to \R^n\), \(x \mapsto x - x_0\), be the translation by \(-x_0\). Then \(S_{*}\mu\) has center of mass at the origin and \(\norm{S_{*}\mu} = \norm{\mu}\).
  Since \(\abs{x_0} < 2\epsilon\), observe that \(B^n(0,\epsilon) \subset B^n(x_0,3\epsilon)\) and \(3\epsilon < \frac{3}{8} < 1\). Then
  \[S_{*}\mu (B^n(0,3\epsilon)) = \mu(B^n(x_0,3\epsilon)) \geq \mu(B^n(0,\epsilon)) \geq \frac{3}{4}\norm{\mu}.\]
  Applying Lemma~\ref{lemma:scaling-mass-redistribution} to \(S_{*}\mu\) with \(3\epsilon\) and \(\gamma\), gives a \(\delta\), that there exists \(\Lambda \geq \frac{1}{3\epsilon}\) and a scaling map \(T\colon \R^n\to\R^n\), \(x\mapsto \Lambda x\), for which
  \(T_{*}S_{*}\mu\big(B^n(0,\Lambda)\setminus B^n(0,1)\big) = 2\delta\), and
  \(T_{*}S_{*}\mu\big(B^n(z,1/4)) \leq \norm{\mu} - \delta \) for all \(z \in \R^n\).
  
  Setting 
  \( A := S^{-1} \circ T \circ S \)
  and
  \( R := 2\Lambda - 1 \geq 2\epsilon (\Lambda - 1) + \Lambda\),
  we see that 
  \[
  A (B^n(0,1)) = B^n((1 - \Lambda)x_0, \Lambda) \subset B^n(0,R)\quad \text{and}\quad
  \norm{A_{*}\mu} = \norm{T_{*}S_{*}\mu} = \norm{\mu}.
  \]
  Since
  \begin{align*}
  \norm{A_{*}\mu} &\geq A_{*}\mu (B^n(0,R)) \geq A_{*}\mu (A(B^n(0,1))) 
  = \mu (B^n(0,1)) = \norm{\mu} = \norm{A_{*}\mu},
  \end{align*}
  we have that \(A_{*}\mu\big (B^n(0,R)\setminus A(B^n(0,1))\big) = 0\).  Thus (1) follows.

  Since \(\abs{x_0} \leq 2\epsilon\), observe that \(B^n(0,1) \subset B^n(x_0, 1+2\epsilon)\). Hence 
  \[ 
  B^n(0,\Lambda) \setminus B^n(x_0, 1+2\epsilon) \subset B^n(0,\Lambda) \setminus B^n(0,1)
  \]
  and 
  \[
  T_*S_*\mu (B^n(0,\Lambda) \setminus B^n(x_0, 1+2\epsilon)) \leq T_*S_*\mu (B^n(0,\Lambda) \setminus B^n(0, 1)) = 2\delta.
  \]
  Thus property (2) follows from \(\supp T_*S_*\mu \subset \bar B^n(0,\Lambda)\) and the inequality
  \begin{align*}
      A_*\mu(B^n(0,R)\setminus B^n(0,1 + 2\epsilon)) &= A_*\mu(A(B^n(0,1))\setminus B^n(0,1 + 2\epsilon)) \\
      &\leq T_*S_*\mu (B^n(0,\Lambda) \setminus B^n(x_0, 1+2\epsilon)) \leq 2\delta.
  \end{align*}

  Since \(z - x_0 \in \R^n\) for all \(z \in \R^n\), we see that 
  \[ A_*\mu (B^n(z,1/4) = T_*S_*\mu(B^n(z - x_0)) \leq \norm{\mu} - \delta.\]
  Property (3) now follows.
\end{proof}

\subsection{Proof of Proposition \ref{prop:global-renormalization}}

As in the statement, let $(\widehat F_j \colon X^{\vee P} \to N)_{j\in \N}$ be an asymptotically $(K,\omega)$-quasiregular sequence converging singularly to a $K$-quasiregular $\omega$-curve $\widehat F\colon X^{\vee P} \to N$ which is a nodal pre-resolution of the sequence \((F_k \colon X \to N)_{k\in\N}\). Also set $\widehat \mu = \mlim^\omega_{j\to \infty} \widehat F_j = \lim_{k\to \infty} \star \widehat F^*_j \omega$ and 
recall that \(P = (P_M)_{M \in \Strata(X)}\), where 
\[P_M = \Sing(\mlim^\omega_{k\to\infty} F_k|_M) \quad \text{for} \ M \in \Strata(X),\]
is a family of discrete sets.
As before, we denote by $\sigma \colon \bS^n \to \overline{\R^n}$ the stereographic projection satisfying $\sigma(e_{n+1}) = 0$ and $\sigma(-e_{n+1}) = \infty$; here $\overline{\R^n} = \R^n \cup \{\infty\}$. 

As a preliminary step, we show that we may assume each set $\widehat P_S = \Sing(\widehat \mu \llcorner S) = \Sing(\mlim^\omega_{j \to \infty} \widehat F_j|_{S})$ is either empty or consists of the north pole $e_{n+1}$ of the bubble $S \in \Strata_X(X^{\vee P})$. 
For this, let $M\in \Strata(X)$ be a stratum of $X$. 
We may assume that, for each $p\in P_M$ and $\tilde p\in \bS^n_p$ satisfying $\widehat \mu(\{\tilde p\}) > (9/10)\widehat \mu(\bS^n_p)$, we have $\tilde p = e^p_{n+1}:= e_{n+1}\in \bS^n$. Indeed, for each $p\in P_M$, there exists at most one point $\tilde p\in \bS^n_p$ satisfying this condition. 
For $p\in P_M$, let $\tilde p\in \bS^n_p$ be this unique point if such exists; otherwise we set $\tilde p = e_{n+1}\in \bS^n_p$ if such a point does not exist. 
For each \(p \in p_M\), let \(T_p \colon \overline{\R^n} \to \overline{\R^n}\) be the translation $x\mapsto x + \sigma(\tilde p)$ and let $\rho_p \colon \overline{\R^n} \to \overline{\R^n}$ be a scaling $x\mapsto \Lambda_p x$ for some \(\Lambda_p > 0\) for which
\begin{equation}
\label{eq:almost-all-mass}
(\rho_p \circ T_p \circ \sigma)_*\widehat \mu(B^n(0,1/10)) > (99/100)\widehat \mu(\bS^n_p).
\end{equation}
Then let $\tau_p \colon \bS^n_p \to \bS^n_p$ be the conjugation $\tau_p = \sigma^{-1} \circ \rho_p \circ T_p \circ \sigma$. Now we may define a \(1\)-quasiconformal homeomorphism $\tau_M \colon M^{\vee P_M} \to M^{\vee P_M}$ by the formulas $\tau_M|_M = \id$ and $\tau|_{\bS^n_p} = \tau_p$ for $p\in P_M$.

Since the nodal manifold $M^{\vee P_M}$ is a nodal submanifold of $X^{\vee P}$ and $X^{\vee P} = (\coprod_{M\in \Strata(X)} M^{\vee P_M})\big/{\sim}$, the mapping $\tau \colon X^{\vee P} \to X^{\vee P}$, given by the formula $\tau|_{M^{\vee P_M}} = \tau_M$ for each $M\in \Strata(X)$, is well-defined. Since the maps $\tau_M$ are $1$-quasiconformal homeomorphisms, so is $\tau$. Furthermore, 
\begin{align*}
\left( \mlim^\omega_{j\to \infty} (\widehat F_j \circ \tau) \right)(\{e^p_{n+1}\}) &= \left( \lim_{j\to \infty} (\widehat F_j \circ \tau_p)^*\omega\right)(\{e^p_{n+1}\}) \\
&= \left( \lim_{j\to \infty} \tau_p^* (\widehat F_j)^* \omega \right)(\{e^p_{n+1}\}) \\
&=\left( (\tau_p^{-1})_*(\mlim^\omega_{j\to \infty}\widehat F_j)\right)(\{e^p_{n+1}\}) \\
&= (\mlim^\omega_{j\to \infty} \widehat F_j))(\tau_p\{e_{n+1}\}) = \widehat \mu( \{\tilde p\}).
\end{align*}
Since $\tau \colon X^{\vee P} \to X^{\vee P}$ is $1$-quasiconformal, we may -- by passing to a subsequence of $(\widehat F_j \circ \tau)_{j\in \N}$ if necessary -- assume that the sets $\widehat P_S$ are either empty or consist solely of the north poles $e_{n+1}$ of bubbles $S \in \Strata_X(X^{\vee P})$ and that the measure $\widehat\mu$ in each bubble $S \in \Strata_M(M^{\vee P_M})$ is concentrated near the north pole in the sense of \eqref{eq:almost-all-mass}.

We now construct a sequence $(h_j \colon X^{\vee P} \to X^{\vee P})_{k\in \N}$ stratum-wise, as we did for the mapping $\tau$. 
Let $M\in \Strata(X)$ and suppose that for $p\in P_M$ we have $\widehat \mu(\{\tilde p\})>(9/10)\widehat \mu(\{\tilde p\})$. In this case, we first define for each $j\in \N$, a measure $\mu^p_j = \star (\widehat F_j \circ \sigma_p^{-1})^* \omega \llcorner \bar B^n$. Since $(\mu^p_j)_{j\in \N}$ converges weakly to the measure $\left( (\sigma_p^{-1})_*\widehat \mu \right)\llcorner \bar B^n$, there exists $j_0\in \N$ for which $\mu^p_j(B^n(0,1/10)) > (95/100)\mu^p_j(\R^n)$ for each $j\ge j_0$. Thus, by Proposition \ref{prop:affine-mass-redistribution}, there exists a conformal (affine) map $A^p_j \colon \R^n \to \R^n$ satisfying both
\begin{equation}
\label{eq:A-mu-bound-1}
(A^p_j)_*\mu^p_j(\R^n\setminus B^n(0,12/10)) \le (2/10)\norm{\mu^p_j}
\end{equation}
and
\begin{equation}
\label{eq:A-mu-bound-2}
(A^p_j)_*\mu^p_j(B^n(x,1/4)) \le (89/100)\norm{\mu^p_j}.
\end{equation}
We define $\alpha^p_j = \sigma_p^{-1} \circ A^p_j \circ \sigma_p^{-1} \colon \bS^n_p \to \bS^n_p$. 

We are now ready to define $h_j \colon X^{\vee P} \to X^{\vee P}$ by the formulas $h_j|_X = \id$, $h_j|_{\bS^n_p} = (\alpha^p_j)^{-1}$ if $\widehat \mu(\{\tilde p\}) > (9/10) \widehat \mu(\bS^n_p)$, and $h_j|_{\bS^n_p} = \id$ otherwise. Also let $\widetilde F_j = \widehat F_j \circ h_j \colon X^{\vee P} \to N$ for each $j\in \N$.
Since Proposition \ref{prop:norm-fams-hold-qrcs} ensures the pre-compact orbits in each stratum, we have, by Proposition~\ref{prop:concentration-of-mass}, that the sequence $(\widetilde F_j)_{j\in \N}$ has a singularly converging subsequence $(\widetilde F_{j_\ell})_{\ell\in \N}$.
Clearly $(\widetilde F_{j_\ell})_{\ell \in \N}$ is a nodal pre-resolution of $(F_k)_{k \in \N}$.

It remains to verify \eqref{eq:split-of-atom}. Let $\widetilde \mu = \mlim^\omega_{\ell\to \infty} \widetilde F_{j_\ell}$. Let $M\in \Strata(X)$ and $p\in P_M$. For $q=-e_{n+1}\in \bS^n_p$, we have, by \eqref{eq:A-mu-bound-1}, that 
\[
\widetilde \mu(\{q\}) \le (2/10) \norm{\mu^p_j} + (1/100)\widehat \mu(\bS^n_p) < (9/10)\widehat \mu(\bS^n_p).
\]
Similarly, by \eqref{eq:A-mu-bound-2}, we have for $q \in \bS^n_p\setminus \{-e_{n+1}\}$, that $\widetilde \mu (\{q\}) < (9/10) \widehat \mu(\bS^n_p)$. The claim follows.
\qed


\section{Proof of the main theorem: Existence of nodal resolutions}

Using nodal resolutions, we are now ready to state and prove our main theorem. Observe that it is a reformulation of Theorem~\ref{thm:main-reconstruction} using our nodal resolutions formalism.

\begin{theorem}\label{prop:bubble-induction}
    Let \(2 \le n \le m\), $K\ge 1$, let \(X\) be a connected, oriented and Riemannian nodal \(n\)-manifold, let \((N,\omega)\) be an \(n\)-calibrated Riemannian \(m\)-manifold having bounded geometry. 
    Let \((F_k \colon X \to N)_{k \in \N}\) be a locally equibounded sequence of \(K\)-quasiregular \(\omega\)-curves for which there exists \(x_0 \in X\) such that the orbit \(\set{F_k(x_0) \mid k \in \N}\) has compact closure.
    Then there exists a subsequence \((F_{k_j})_{j\in \N}\) of \((F_k)_{k\in \N}\) which has a nodal resolution \((\widehat F_\ell \colon \widehat X \to N)_{\ell\in \N}\), where \(\widehat X\) is a bubble tree over \(X\) having the property that each component of $\widehat X\setminus X$ has finitely many strata. 
\end{theorem}

\begin{proof}
Recall that by Theorem~\ref{thm:Ikonen-Holder}, \(N\) has a small energy bound \(E_N > 0\). 
Let \(\epsilon_0 = \epsilon_0(N, \bS^n,K,E_N,\omega) > 0\) be the energy gap from Theorem~\ref{thm:analytic-energy-gap} for \(K\)-quasiregular \(\omega\)-curves \(\bS^n \to N\).
Since $(F_k)_k$ is a sequence of $K$-quasiregular $\omega$-curves, it trivially admits a \((K,\omega)\)-quasiregular exhaustion about \(\emptyset \subset X\). Since \(N\) is complete and \(X\) is connected, by the Hopf--Rinow theorem for each \(x \in X\), the orbit \(\set{F_k(x) \mid k \in \N}\) has compact closure in \(N\); see the proof of Proposition~\ref{prop:norm-fams-hold-qrcs} for a similar argument. 
Applying Proposition~\ref{prop:concentration-of-mass}, with data $Q=\emptyset$ and \(E = K\epsilon_0\),  gives the existence of an asymptotically $(K,\omega)$-quasiregular subsequence \((F_{k_j})_{j\in\N}\) of \((F_k)_k\) and which converges $\epsilon_0$-singularly to a \(K\)-quasiregular \(\omega\)-curve \(F \colon X \to N\).

We begin to resolve the singularities of $(F_{k_j})_j$. Set \(\mu = \mlim^\omega_{j\to \infty}F_{k_j}\). Then \(\Sing(\mu)\) is a countable discrete set. 
Set also \(\widetilde F_j^{(0)} := F_{k_j}\) for $j\in \N$, \(\widetilde \mu_0 = \mlim^\omega_{j\to \infty} \widetilde F_j^{(0)}\), 
\[
P_0 = \left( \Sing(\mlim^\omega_{j\to \infty} \widetilde F_j^{(0)}|_M)\right)_{M\in \Strata(X)},
\]
and \(\widetilde X_0 := X\).
Then, by Propositions~\ref{prop:first-surgery-nodal-manifold} and \ref{prop:global-renormalization}, there exist a bubble tree \(\widetilde X_1 = \widetilde X_0^{\vee P_0}\) over \(\widetilde X_0\), and a locally equibounded and asymptotically \((K,\omega)\)-quasiregular sequence \((\widetilde F^{(1)}_\ell \colon \widetilde X_1 \to N)_\ell\) which is a nodal pre-resolution of \((\widetilde F_j^{(0)})_j\) and for which the measure
\(\widetilde \mu_1 = \mlim^\omega_{\ell \to \infty} \widetilde F^{(1)}_\ell\)
satisfies
\[
  \widetilde \mu_1 (\tilde p) \leq (9/10) \widetilde \mu_0(p) \quad \text{for all}\ \tilde p \in \bS^n_p\ \text{and}\ p \in P_0.
\]

Assuming that our pre-resolution is not already a resolution, we continue to resolve the singularities of $(\widetilde F^{(1)}_\ell)_\ell$, that is,
\[
P_1 = \left( \Sing(\mlim^\omega_{\ell\to \infty} \widetilde F_\ell^{(1)}|_M)\right)_{M\in \Strata(\widetilde X_1)}.
\]
Again, since \(N\) is complete and \(\widetilde X_1\) is connected, by the Hopf--Rinow theorem for each \(x \in \widetilde X_1\), the orbit \(\set{\widetilde F^{(1)}_\ell(x) \mid \ell \in \N}\) has compact closure in \(N\).
Since \((\widetilde F^{(1)}_\ell)_{\ell \in \N}\) is a nodal pre-resolution of  \((\widetilde F^{(0)}_\ell)_{\ell \in \N}\), it converges locally uniformly in a neighborhood of \(\Sing_{\widetilde X_0}(\widetilde X_1)\), and hence we may apply Proposition~\ref{prop:concentration-of-mass} and continue as before.

Now, proceeding by induction, we obtain, for each $k\in \N$, a bubble tree \(\widetilde X_k = \widetilde X_{k-1}^{\vee P_{k-1}}\) over \(\widetilde X_{k-1}\) (and over \(X\)), and an asymptotically \((K,\omega)\)-quasiregular sequence \((\widetilde F^{(k)}_\ell \colon \widetilde X_k \to N)_\ell\), which is a nodal pre-resolution of \((\widetilde F_j^{(k-1)})_j\) (and hence
also of \((F_{k_j})_j\)). As before, the sequence $(\widetilde F^{(k)}_\ell)_\ell$ is locally equibounded, and satisfies
\[
\widetilde \mu_k (\hat p) \leq (9/10) \widetilde \mu_{k-1}(\pi_{\widetilde X_k,\widetilde X_{k-1}}(\hat p)) \leq \cdots \leq (9/10)^k \mu(p)
\]
for each $\hat p \in \pi_{\widehat X_k,X}^{-1}(p)$ and $p\in \Sing(\mu)$, where \(\widetilde \mu_k = \mlim^\omega_{\ell \to \infty} \widetilde F^{(k)}_\ell\).

Since $\widetilde X_{k-1}$ is a nodal submanifold of $\widetilde X_k$ for each $k\in \N$, the union $\widetilde X = \bigcup_{k\in \N} \widetilde X_k$ is a well-defined nodal manifold. We denote 
\[
\widetilde G^{(k)}_j = \widetilde F^{(k)}_j \circ \pi_{\widetilde X,\widetilde X_k} \colon \widetilde X \to N
\]
for each $k\in \N$ and $j\in \N$. Since $\widetilde X$ consists of countably many strata and, for each bubble $S$ of $\widetilde X$, there exists a subsequence $(\widetilde G^{(k_\ell)}_{j_\ell}|_{S})_\ell$ of $(\widetilde G^{(k)}_j)_{k,j}$ converging locally uniformly to a $K$-quasiregular $\omega$-curve $S\to N$, we may pass to a diagonal subsequence $(\widetilde G_\ell \colon \widetilde X\to N)_\ell$ of $(\widetilde F^{(k)}_j)_{k,j}$, which converges locally uniformly on each stratum of $\widetilde X$ to a $K$-quasiregular $\omega$-curve. Since $\widetilde X$ is a locally finite nodal manifold, i.e.~each point has a neighborhood which meets only finitely many strata, we conclude that $(\widetilde G_\ell)_\ell$ converges locally uniformly to $K$-quasiregular $\omega$-curve $\widetilde G\colon \widetilde X\to N$. 

We show next that the measure $\widetilde \mu = \mlim^\omega_{\ell \to \infty} \widetilde G_\ell$ has no atoms. Let $M\in \Strata(X)$ and $p\in P_M$. Since \(\mu(\{p\})<\infty\), there exists \(k_p \in \N\) for which \((9/10)^{k_p} \mu(p) < \epsilon_0/K\). Thus, for each \(S \in \Strata_{\widetilde X_{k_p- 1}}(\widetilde X_{k_p})\), we have that
\[
\int_{S} \widetilde G^*\omega = \lim_{\ell \to \infty} (\mlim^\omega \widetilde G_\ell)(S) \le (9/10)^{k_p} \mu(p) < \epsilon_0/K.
\]
Hence
\[
\int_S \norm{D\widetilde G}^n \le K \int_{S} \widetilde G^*\omega < \epsilon_0.
\]
Thus, by Theorem \ref{thm:analytic-energy-gap}, the restriction $\widetilde G|_S$ is a constant map. Thus, for the restriction $\widetilde G|_{\overline{Y}} \colon \overline{Y} \to N$, where $Y \subset \widetilde X\setminus \widetilde X_{k_p}$ is the component contained in $\pi_{\widetilde X,X}^{-1}(p)$, we have that
\[
\widehat \mu_{k_p+1}(\pi_{\widetilde X, \widetilde X_{k_p}}^{-1}(\overline{Y})) = \int_{\overline{Y}} \widetilde G^*\omega = 0. 
\]
Hence $\widetilde \mu_{k_p}$ has no atoms in the bubbles in $\widetilde X_{k_p}\cap \pi_{\widetilde X, X}^{-1}(p)$. Since we have $\mlim^\omega_{\ell\to \infty} \widetilde G_\ell|_{\widetilde X_k} = \widetilde \mu_k$ for each $k$, we conclude that $\widetilde \mu$ has no atoms. 

Finally, by removing the bubbles which are not contained in $\pi_{\widetilde X,X}^{-1}(p) \cap X_{k_p}$ for each $p\in \Sing(\mu)$, we obtain a nodal manifold $\widehat X\subset \widetilde X$ having finite bubble trees over each point $p\in \Sing(\mu)$. We define $\widehat F_\ell = \widetilde G_\ell|_{\widehat X} \colon \widehat X\to N$ for each $\ell\in \N$. Since the bubbles removed from $\widetilde X$ have zero $\mu$-measure, we have that the map $\widehat F = \widetilde G|_{\widehat X} \colon \widehat X\to N$ and the sequence $(\widehat F_\ell)_{\ell\in \N}$, satisfy the wanted properties. This concludes the proof.
\end{proof}

\section{Reformulations of the main theorem}
\label{sec:GH}

We reformulate Theorem \ref{thm:main-reconstruction} in two ways. The first reformulation is in terms of Gromov--Hausdorff convergence.
Recall that a mapping $\varphi\colon X\to Y$ between metric spaces $X$ and $Y$ is an \emph{$\varepsilon$-rough isometry for $\varepsilon\ge 0$} if 
\[
d_X(x,x') - \varepsilon \le d_Y(\varphi(x),\varphi(x')) \le d_X(x,x') + \varepsilon
\]
for $x,x'\in X$ and the image $fX$ is $\varepsilon$-dense in $Y$, that is, $B_Y(\varphi X,\varepsilon)=Y$. 

For the statement, we call a pair $(X,f\colon X\to Y)$, where $X$ and $Y$ are metric spaces, a \emph{mapping package}. A sequence $(X_k,f_k \colon X_k \to Y)_{k\in \N}$ of mapping packages \emph{Gromov--Hausdorff converges to a mapping package $(X,f)$, denoted $(X_k,f_k) \GHto (X,f)$}, if there exists a sequence $\varepsilon_k \to 0$ and, for each $k\in \N$, an $\varepsilon_k$-rough isometry $\psi_k \colon X \to X_k$ for which the sequence $(f_k \circ \psi_k \colon X\to Y)_{k\in\N}$ converges locally uniformly to $f$ as $k\to \infty$.
Recall that the existence of such sequence $(\psi_k \colon X \to X_k)_{k\in \N}$ is equivalent to Gromov--Hausdorff convergence of spaces $(X_k)_{k\in \N}$ to $X$.
We refer to Burago, Burago, and Ivanov \cite[Chapter 7]{Burago-Burago-Ivanov--book} and to Bridson and Haefliger \cite[Chapter I.5]{Bridson-Haefliger--book} for detailed discussions on Gromov--Hausdorff convergence.

In the following version of Theorem \ref{thm:main-reconstruction}, both $d_{g}$ and $d_X$ denote the length metrics of the nodal Riemannian manifold $(X,g)$.

\begin{theorem}[Gromov--Hausdorff convergence]
\label{thm:main-GH}
Let $2\le n \le m$, let $X$ be a connected, oriented Riemannian nodal $n$-manifold, let $(N,\omega)$ be an $n$-calibrated Riemannian $m$-manifold having bounded geometry, and $K\ge 1$. Let $(F_k \colon X\to N)_{k\in \N}$ be a locally equibounded sequence of $K$-quasiregular $\omega$-curves for which there exists \(x_0 \in X\) such that the orbit \(\set{F_k(x_0) \mid k \in \N}\) has compact closure.
Then there exists a subsequence $(F_{k_j})_{j\in \N}$ of $(F_k)_{k\in\N}$, a sequence of Riemannian metrics $(g_j)_{j\in \N}$ on $X$, a bubble tree $\widehat X$ over $X$, and a $K$-quasiregular $\omega$-curve $\widehat F \colon \widehat X \to N$ for which
\[
(X,d_{g_j},F_{k_j}) \GHto (\widehat X, d_{\widehat X}, \widehat F)
\]
as mapping packages, where $d_{g_j}$ is the distance of $g_j$ and $d_{\widehat X}$ the distance function of the bubble tree $\widehat X$.
\end{theorem}

The other reformulation of Theorem \ref{thm:main-reconstruction} is in terms of nodal (or tight) convergence. Recall that for a bubble tree $\widehat M$ over an $n$-manifold $M$, a map $\rho \colon M \to \widehat M$ is a \emph{pinching map} if 
\begin{enumerate}
\item there exists an indexed family $(B_S)_{S\in \Strata_M(\widehat M)}$ of closed topological $n$-balls  in $M$ for which $S \subset \rho(B_S)$ and $\rho(\partial B_S)$ is a nodal point in $\Sing(\widehat M)$, and
\item $\rho$ is an embedding in each component of $M\setminus \bigcup_{S\in \Strata_M(\widehat M)} \partial B_S$.
\end{enumerate}

\begin{theorem}[Tight convergence]
\label{thm:main-pinching}
Let $2\le n \le m$, let $M$ be a connected, oriented Riemannian $n$-manifold, let $(N,\omega)$ be an $n$-calibrated Riemannian $m$-manifold having bounded geometry, and $K\ge 1$. 
Let $(F_k \colon M\to N)_{k\in \N}$ be a locally equibounded sequence of $K$-quasiregular $\omega$-curves for which there exists \(x_0 \in M\) such that the orbit \(\set{F_k(x_0) \mid k \in \N}\) has compact closure.
Then there exists a subsequence $(F_{k_j})_{j\in \N}$ of $(F_k)_{k\in\N}$, a bubble tree $\widehat M$ over $M$, a sequence of pinching maps $(\rho_j \colon M \to \widehat M)_{j\in \N}$ on $M$,  and a $K$-quasiregular $\omega$-curve $\widehat F \colon \widehat M \to N$ for which
\[
F_{k_j} \circ \rho_j^{-1}|_{\widehat M \setminus \Sing(\widehat M)} 
\to \widehat F|_{\widehat M \setminus \Sing(\widehat M)} 
\]
locally uniformly.
\end{theorem}

Both of these statements are based on the following well-known observation.
\begin{lemma}
\label{lemma:un-pinching}
Let $X$ be a nodal Riemannian manifold and $\varepsilon>0$. Then there exists a Riemannian manifold $Y$ and an $\varepsilon$-rough isometry $\rho \colon Y \to X$ for which there exists an $\varepsilon$-dense open submanifold $\Omega \subset X$ and a locally isometric embedding $\iota \colon \Omega\to Y$ for which the composition $\rho\circ \iota \colon \Omega \to X$ is identity. Moreover, if each $p\in \Sing(X)$ meets at most two strata of $X$, we may take $\rho$ to be an embedding in the components of $Y\setminus \rho^{-1}(\Sing(X))$ and each $\rho^{-1}(p)$ to be an $n$-sphere for $p\in \Sing(X)$.
\end{lemma}

Since the existence of such a manifold $Y$ is discussed in many proofs for different versions of Gromov's theorem, we merely recall an idea for the construction; see e.g.~Parker \cite{Parker-JDG-1996}. 

\begin{proof}[Sketch of Proof of Lemma \ref{lemma:un-pinching}]
Let $X$ be a nodal Riemannian manifold and let $\delta>0$ to be determined later. The manifold $Y$ is constructed by independently removing each singular point of $X$ as follows.

Let $p\in \Sing(X)$ be a singular point and let $M_1,\ldots, M_{m_p}$ be the strata of $X$ containing $p$. Then there exists $r_p\in (0,\varepsilon/4)$ for which each exponential map $\exp_p \colon B_{T_p M_i}(0,2r_p) \to B_{M_i}(p,2r_p)$ is $(1+\delta)$-bilipschitz. By passing to a smaller radius $r_p$ if necessary, we may assume that there exists mutually disjoint balls $B_S(x_1,2r_p),\ldots, B_S(x_{m_p},2r_p)$ on the standard $n$-sphere $S=\bS^n(\varepsilon/(4\pi))$ of radius $\varepsilon/(4\pi)$.

Next we remove the open sets $U_X = B_{M_1}(p,r_p) \cup \cdots \cup B_{M_{m_p}}(p,r_p)$ and $U_S = B_S(x_1,r_p)\cup \cdots \cup B_S(x_{m_p},r_p)$ from $X$ and $S$, respectively, and glue the spheres $\partial B_{M_i}(p,r_p)$ and $\partial B_S(x_i,r_p)$ together with the composition of the exponential mappings. By taking $\delta$ small enough, the obtained manifold $Y$ has a natural smooth structure and has a length metric $d$ for which $X\setminus U_X\subset Y$ and $S\setminus U_X \subset Y$ are smooth submanifolds with boundary and the metric space $(Y,d)$ has Gromov-Hausdorff distance at most $\varepsilon/2$ to $X$. Since $Y$ has a Riemannian metric in $X\setminus U_X$ and $S\setminus U_S$, we have that -- by taking $\delta$ even smaller -- we may smooth these Riemannian metrics to obtain a Riemannian metric $g_Y$ on $Y$ for which the mapping $\rho \colon Y \to X$, which is identity on $V_X=B_{M_1}(p,2r_p)\cup \cdots \cup B_{M_{m_p}}(p,2r_p)$ and collapses the complement of $U_S$ to the point $p$, is an $\varepsilon$-rough isometry. 

If the singular points of $X$ meet only two strata, we may replace the balls $B_S(p,2r_p)$ and $B_S(p,2r_p)$ by hemispheres. This yields the wanted additional property. 
\end{proof}

\begin{corollary}
\label{cor:un-pinching}
Let $\widehat X$ be a bubble tree over a Riemannian manifold $M$ and $\varepsilon>0$. Then there exists a Riemannian metric $g$ on $M$, an $\varepsilon$-rough isometry $\rho \colon (M,g) \to \widehat X$, and an $\varepsilon$-dense open set $\Omega \subset M$ for which $\rho|_\Omega \colon \Omega \to \widehat X$ is the identity. Moreover, if each singular point $p\in \Sing(\widehat X)$ meets at most two strata of $\widehat X$, then we may take $\rho \colon M\to \widehat X$ to be, in addition, a pinching map.
\end{corollary}
\begin{proof}
Since $\widehat X$ is a (locally finite) bubble tree over $M$, we have that manifold $Y$ constructed in the proof of Lemma \ref{lemma:un-pinching} is an iterated connected sum of $M$ with $n$-spheres and hence diffeomorphic to $\widehat X$. By the proof of Lemma \ref{lemma:un-pinching}, the preimage $\rho^{-1}(p)$ of each $p\in \Sing(\widehat X)$ under $\rho$ in Lemma \ref{lemma:un-pinching} is an $n$-sphere bounding a topological $n$-ball in $M$, the claim now follows from Lemma \ref{lemma:un-pinching} and its proof.
\end{proof}

Having Corollary \ref{cor:un-pinching} at our disposal, Theorems \ref{thm:main-GH} and \ref{thm:main-pinching} are almost immediate consequences of Theorem \ref{thm:main-reconstruction}. For this reason, we only indicate the steps.

\begin{proof}[Sketches of proofs of Theorems \ref{thm:main-GH} and \ref{thm:main-pinching}]
For the proof of Theorem \ref{thm:main-GH}, let $\widehat X$ be a bubble tree over $X$, $(F_{k_j})_{j\in \N}$ a subsequence of $(F_k)_{k\in\N}$, and let $\widehat F \colon \widehat X\to N$ as in Theorem \ref{thm:main-reconstruction}. Let also $(\varepsilon_j)_{j\in \N}$ be a positive sequence tending to zero.
By the first part of Corollary \ref{cor:un-pinching}, we have a sequence of Riemannian metrics $(g_j)_{j\in\N}$ on $X$ and $\varepsilon_j$-rough isometries $\rho_j \colon (X,d_{g_j}) \to \widehat X$. The convergence $(X,g_j,F_{k_j}) \GHto (\widehat X,\widehat F)$ of mapping packages follows now by fixing a sequence $(\psi_j \colon \widehat X \to X)_{j\in\N}$ of rough inverses of maps $\rho_j$ and observing that $F_{k_j} \circ \psi_j \to \widehat F$ locally uniformly as $j\to \infty$. This concludes the sketch of a proof of Theorem \ref{thm:main-GH}.

For a proof of Theorem \ref{thm:main-pinching}, we consider a sequence $(F_k \colon M \to N)_{k\in \N}$ of $K$-quasiregular $\omega$-curves on a Riemannian $n$-manifold $M$ instead of a nodal manifold $X$. The proof is now analogous to the argument above using the additional information provided by Corollary \ref{cor:un-pinching} that $\varepsilon_j$-rough isometries $\rho_j \colon (X,g_j) \to \widehat X$ can be taken to be pinching maps.
\end{proof}

\bibliographystyle{abbrv}
\bibliography{Gromov-compactness-Pankka-Pim--arXiv}

\begin{thebibliography}{10}

\bibitem{Pansu-1994}
M.~Audin and J.~Lafontaine, editors.
\newblock {\em Holomorphic curves in symplectic geometry}, volume 117 of {\em
  Progress in Mathematics}.
\newblock Birkh\"auser Verlag, Basel, 1994.

\bibitem{bojarski1983analytical}
B.~Bojarski and T.~Iwaniec.
\newblock Analytical foundations of the theory of quasiconformal mappings in
  {${\bf R}\sp{n}$}.
\newblock {\em Ann. Acad. Sci. Fenn. Ser. A I Math.}, 8(2):257--324, 1983.

\bibitem{BEHWZ-2003}
F.~Bourgeois, Y.~Eliashberg, H.~Hofer, K.~Wysocki, and E.~Zehnder.
\newblock Compactness results in symplectic field theory.
\newblock {\em Geom. Topol.}, 7:799--888, 2003.

\bibitem{Bridson-Haefliger--book}
M.~R. Bridson and A.~Haefliger.
\newblock {\em Metric spaces of non-positive curvature}, volume 319 of {\em
  Grundlehren der mathematischen Wissenschaften [Fundamental Principles of
  Mathematical Sciences]}.
\newblock Springer-Verlag, Berlin, 1999.

\bibitem{Broder-Iliashenko-Madnick}
K.~{Broder}, A.~{Iliashenko}, and J.~{Madnick}.
\newblock {Hyperbolicity and Schwarz Lemmas in Calibrated Geometry}.
\newblock {\em arXiv e-prints}, page arXiv:2507.16313, July 2025.

\bibitem{Burago-Burago-Ivanov--book}
D.~Burago, Y.~Burago, and S.~Ivanov.
\newblock {\em A course in metric geometry}, volume~33 of {\em Graduate Studies
  in Mathematics}.
\newblock American Mathematical Society, Providence, RI, 2001.

\bibitem{Buser-Karcher-book}
P.~Buser and H.~Karcher.
\newblock {\em Gromov's almost flat manifolds}, volume~81 of {\em
  Ast\'erisque}.
\newblock Soci\'et\'e{} Math\'ematique de France, Paris, 1981.

\bibitem{Chavel-book}
I.~Chavel.
\newblock {\em Riemannian geometry}, volume~98 of {\em Cambridge Studies in
  Advanced Mathematics}.
\newblock Cambridge University Press, Cambridge, second edition, 2006.
\newblock A modern introduction.

\bibitem{Cheng-Karigiannis-Madnick}
D.~R. Cheng, S.~Karigiannis, and J.~Madnick.
\newblock Bubble tree convergence of conformally cross product preserving maps.
\newblock {\em Asian J. Math.}, 24(6):903--984, 2020.

\bibitem{Cheng-Karigiannis-Madnick--variational}
D.~R. Cheng, S.~Karigiannis, and J.~Madnick.
\newblock A variational characterization of calibrated submanifolds.
\newblock {\em Calc. Var. Partial Differential Equations}, 62(6):Paper No. 174,
  38, 2023.

\bibitem{Dyer-Vegter-Wintraecken}
R.~Dyer, G.~Vegter, and M.~Wintraecken.
\newblock Riemannian simplices and triangulations.
\newblock {\em Geom. Dedicata}, 179:91--138, 2015.

\bibitem{Gromov-Invent-1985}
M.~Gromov.
\newblock Pseudo holomorphic curves in symplectic manifolds.
\newblock {\em Invent. Math.}, 82(2):307--347, 1985.

\bibitem{harvey1982calibrated}
R.~Harvey and H.~B. Lawson, Jr.
\newblock Calibrated geometries.
\newblock {\em Acta Math.}, 148:47--157, 1982.

\bibitem{Heikkila-PLMS}
S.~Heikkil\"a.
\newblock Quasiregular curves and cohomology.
\newblock {\em Proc. Lond. Math. Soc. (3)}, 128(5):Paper No. e12602, 27, 2024.

\bibitem{heikkila2023rham}
S.~Heikkil\"a and P.~Pankka.
\newblock De {R}ham algebras of closed quasiregularly elliptic manifolds are
  {E}uclidean.
\newblock {\em Ann. of Math. (2)}, 201(2):459--488, 2025.

\bibitem{heinonen2018nonlinear}
J.~Heinonen, T.~Kilpel\"{a}inen, and O.~Martio.
\newblock {\em Nonlinear potential theory of degenerate elliptic equations}.
\newblock Dover Publications, Inc., Mineola, NY, 2006.
\newblock Unabridged republication of the 1993 original.

\bibitem{HKSTsobolevSpaces}
J.~Heinonen, P.~Koskela, N.~Shanmugalingam, and J.~T. Tyson.
\newblock {\em Sobolev spaces on metric measure spaces}, volume~27 of {\em New
  Mathematical Monographs}.
\newblock Cambridge University Press, Cambridge, 2015.
\newblock An approach based on upper gradients.

\bibitem{Hummel-book-1997}
C.~Hummel.
\newblock {\em Gromov's compactness theorem for pseudo-holomorphic curves},
  volume 151 of {\em Progress in Mathematics}.
\newblock Birkh\"auser Verlag, Basel, 1997.

\bibitem{ikonen2023pushforward}
T.~Ikonen.
\newblock Pushforward of currents under {S}obolev maps.
\newblock {\em arXiv pre-print 2303.15003}, 2023.

\bibitem{ikonen2024remove}
T.~Ikonen.
\newblock Quasiregular {C}urves: {R}emovability of {S}ingularities.
\newblock {\em J. Geom. Anal.}, 35(10):Paper No. 315, 2025.

\bibitem{Ikonen-Pankka-2024}
T.~{Ikonen} and P.~{Pankka}.
\newblock {Liouville's theorem in calibrated geometries}.
\newblock {\em arXiv e-prints}, page arXiv:2410.02722, Oct. 2024.

\bibitem{Ivrii}
O.~Ivrii.
\newblock On meromorphic functions whose image has finite spherical area.
\newblock {\em Studia Math.}, 272(2):121--137, 2023.

\bibitem{Joyce-book-2007}
D.~D. Joyce.
\newblock {\em Riemannian holonomy groups and calibrated geometry}, volume~12
  of {\em Oxford Graduate Texts in Mathematics}.
\newblock Oxford University Press, Oxford, 2007.

\bibitem{kangasniemi2021notes}
I.~{Kangasniemi}.
\newblock {Notes on quasiregular maps between Riemannian manifolds}.
\newblock {\em arXiv e-prints}, page arXiv:2109.01638, Sept. 2021.

\bibitem{Luisto-Prywes}
R.~Luisto and E.~Prywes.
\newblock Open and discrete maps with piecewise linear branch set images are
  piecewise linear maps.
\newblock {\em J. Lond. Math. Soc. (2)}, 103(3):1186--1207, 2021.

\bibitem{Meier-Vikman-Wenger-2025-bubbling-harmonic-2-spheres}
D.~{Meier}, N.~{Vikman}, and S.~{Wenger}.
\newblock {Energy minimizing harmonic 2-spheres in metric spaces}.
\newblock {\em arXiv e-prints}, page arXiv:2503.08553, Mar. 2025.

\bibitem{Morgan-book-2016}
F.~Morgan.
\newblock {\em Geometric measure theory}.
\newblock Elsevier/Academic Press, Amsterdam, fifth edition, 2016.
\newblock A beginner's guide, Illustrated by James F. Bredt.

\bibitem{Onninen-Pankka-qrc-higher-integrability}
J.~Onninen and P.~Pankka.
\newblock Quasiregular curves: {H}\"older continuity and higher integrability.
\newblock {\em Complex Anal. Synerg.}, 7(3):Paper No. 26, 9, 2021.

\bibitem{Pang-Nevo-Zalcman}
X.~Pang, S.~Nevo, and L.~Zalcman.
\newblock {Quasinormal families of meromorphic functions}.
\newblock {\em Revista Matemática Iberoamericana}, 21(1):249–262, 2005.

\bibitem{pankka2020qrcs}
P.~Pankka.
\newblock Quasiregular curves.
\newblock {\em Ann. Acad. Sci. Fenn. Math.}, 45(2):975--990, 2020.

\bibitem{pankka-souto2023bubbleqr}
P.~Pankka and J.~Souto.
\newblock Bubbling of quasiregular maps.
\newblock {\em Math. Ann.}, 385(3-4):1577--1610, 2023.

\bibitem{Parker-JDG-1996}
T.~H. Parker.
\newblock Bubble tree convergence for harmonic maps.
\newblock {\em J. Differential Geom.}, 44(3):595--633, 1996.

\bibitem{Rickman-book-1993}
S.~Rickman.
\newblock {\em {Quasiregular mappings}}, volume~26 of {\em {Ergebnisse der
  Mathematik und ihrer Grenzgebiete (3) [Results in Mathematics and Related
  Areas (3)]}}.
\newblock Springer-Verlag, Berlin, 1993.

\bibitem{Ruan-Tian-1995}
Y.~Ruan and G.~Tian.
\newblock A mathematical theory of quantum cohomology.
\newblock {\em J. Differential Geom.}, 42(2):259--367, 1995.

\bibitem{Simon-geo-meas-book-1983}
L.~Simon.
\newblock {\em Lectures on geometric measure theory}, volume~3 of {\em
  Proceedings of the Centre for Mathematical Analysis, Australian National
  University}.
\newblock Australian National University, Centre for Mathematical Analysis,
  Canberra, 1983.

\bibitem{Wolfson-1988}
J.~G. Wolfson.
\newblock Gromov's compactness of pseudo-holomorphic curves and symplectic
  geometry.
\newblock {\em J. Differential Geom.}, 28(3):383--405, 1988.

\end{thebibliography}

\end{document}